\documentclass[11pt]{amsart}
\usepackage{amssymb, latexsym, amsmath, amsfonts, tikz}
\usepackage{hyperref, enumitem, fancyhdr}
\usepackage{color}

\newtheorem{thm}{Theorem}[section]
\newtheorem{cor}[thm]{Corollary}
\newtheorem{lem}[thm]{Lemma}
\newtheorem{prop}[thm]{Proposition}
\theoremstyle{definition}

\theoremstyle{remark}
\newtheorem{rem}[thm]{Remark}
\numberwithin{equation}{section}
\theoremstyle{remark}

\newcommand{\mbb}{\mathbb}
\newcommand{\ra}{\rightarrow}

\newcommand{\pa}{\partial}
\newcommand{\ov}{\overline}
\newcommand{\sm}{\setminus}

\newcommand{\no}{\noindent}
\newcommand{\al}{\alpha}

\newcommand{\cal}{\mathcal}

\newcommand{\la}{\lambda}

\newcommand{\La}{\Lambda}

\newcommand{\si}{\sigma}

\begin{document}
\title{Ergodic properties of families of H\'{e}non maps}
\author{Ratna Pal, Kaushal Verma}

\address{Ratna Pal: Department of Mathematics, Indian Institute of Science, Bangalore 560 012, India}
\address{
Current Address: Department of Mathematics, Indian Institute of Science Education and Research, Pune, Maharashtra-411008, India}
	\email{ratna@iiserpune.ac.in}

\address{Kaushal Verma: Department of Mathematics, Indian Institute of Science, Bangalore 560 012, India}
\email{kverma@math.iisc.ernet.in}

\begin{abstract}
Let $\{ H_{\lambda} \}$ be a continuous family of H\'{e}non maps parametrized by $\lambda \ in M$, where $M \subset \mathbb C^k$ is compact. The purpose of this paper is to understand some aspects of the random dynamical system obtained by iterating maps from this family. As an application, we study skew products of H\'{e}non maps and obtain lower bounds for their entropy.
\end{abstract}

\maketitle


\section{Introduction}

\no The purpose of this paper is to study the random dynamical system obtained by iterating H\'{e}non maps that lie in a given compact family of such maps. To make this precise, let $\cal H = \{H_{\la}\}$ be a continuous family of H\'{e}non maps parametrized by $\la \in M$, where $M \subset \mbb C^k$ is compact. Thus, for each $\la \in M$ and $(x, y) \in \mbb C^2$, the map
\[
H_{\la} = H_{\la}(x, y) = H_{\la}^{(m)} \circ H_{\la}^{(m-1)} \circ \cdots \circ H_{\la}^1(x, y)
\]
where for each $1 \le j \le m$, $H_{\la}^{(j)}$ is a generalized H\'{e}non map defined by
\begin{equation}
H_{\la}^{(j)}(x, y) = (y, p_{j, \la}(y) - a_j(\la) x)
\end{equation}
with $p_{j, \la}(y)$ a monic polynomial of degree $d_j \ge 2$ whose coefficients and $a_j(\la)$ are continuous functions on $M$. The degree of $H_{\la}$ is $d = d_1 d_2 \cdots d_m$ which does not vary with $\la \in M$. Let $m$ be normalized Lebesgue measure on $M$. Kolmogorov's extension theorem guarantees the existence of a probability measure, say $\ov m$ on the cartesian product of countably many copies of $M$, i.e.,
\[
\ov m = m \times m \times \ldots
\]
is a probability measure on $X = M \times M \times \ldots$, where $X$ is endowed with the product topology. For each $\La = (\la_1, \la_2, \ldots) \in X$ and $\la \in M$, let
\[
\la \La = (\la, \la_1, \la_2, \ldots) \;\; \text{and} \;\; \La_n = (\la_n, \la_{n+1}, \ldots)
\]
for all $n \ge 1$. Furthermore, let
\begin{equation}
H_{n, \La}^{\pm} = H_{\la_n}^{\pm 1} \circ H_{\la_{n-1}}^{\pm 1} \circ \cdots \circ H_{\la_1}^{\pm 1}
\end{equation}
for each $\La \in X$ and $n \geq 1$. There were two reasons for considering a set-up as general as this. First, the case when $M$ reduces to a point corresponds precisely to the situation when a given single H\'{e}non map is iterated. Much is known about this case and motivated by this, it seemed interesting to understand the `averaged-out' dynamics of the family $\{H_{n, \La}^{\pm}\}$, not for a fixed $\La \in X$, but by allowing $\La$ to vary in $X$. Second, the work of Forn{\ae}ss-Weickert \cite{FWe} must be mentioned here -- they work with a holomorphic family of holomorphic endomorphisms of $\mbb P^k$ and develop analogs of several dynamically interesting objects associated with a single holomorphic endomorphism of $\mbb P^k$. H\'{e}non maps have a point of indeterminacy when viewed as maps from $\mbb P^2$, but they are regular in the sense of Sibony. Indeed, if $[t: x: y]$ are the homogeneous coordinates on $\mbb P^2$, a given H\'{e}non map $H$ defined by
\[
H(x, y) = (y, p(y) - \delta x)
\]
where $p(y)$ is a monic polynomial of degree $l \ge 2$ and $\delta \not=0$, extends to a birational map of $\mbb P^2$ as
\[
[t: x: y] \mapsto [t^l: yt^{l-1} : t^l p(yt^{-1}) - \delta x t^{l-1}],
\]
and this extension will still be denoted by $H$. Note that each of these maps (as $p = p(y), \delta$ vary) have a common attracting fixed point at $I^- = [0:0:1]$, which is the unique indeterminacy point of $H^{-1}$ and a common repelling fixed point at $I^+ = [0:1:0]$ which is the unique indeterminacy point of $H$ in $\mbb P^2$. Further, the backward orbit of $I^+$ and the forward orbit of $I^-$ under a given sequence $(H_n)_{n \ge 1}$ of such H\'{e}non maps are uniformly separated, i.e., $I^+_{\infty} \cap I^-_{\infty} = \emptyset$, where
\[
I^+_{\infty} = \bigcup_{n=1}^{\infty} (H_n^{-1} \circ \cdots \circ H_1^{-1})(I^+)
\]
and
\[
I^-_{\infty} = \bigcup_{n=1}^{\infty} (H_n \circ \cdots \circ H_1)(I^-)
\]
since $I^+_{\infty} = I^+$ and $I^-_{\infty} = I^-$, as can be checked. It is because of this uniform separateness that it is possible to develop analogous results in the random case. With these observations in mind, working with families of H\'{e}non maps sounded like a reasonable choice in an attempt to see how far one could go in obtaining analogs of the results in \cite{FWe}.

\medskip

To describe some of the results in this direction, for each $\La = (\la_1, \la_2, \ldots) \in X$ and $n \ge 1$, let
\[
G_{n, \La}^{\pm}(z) = \frac{1}{d^n} \log^+ \Vert H^{\pm}_{n, \La}(z) \Vert
\]
where, as usual, $\log^+t = \max\{ \log t, 0\}$. For each $\La \in X$, the sets $I^{\pm}_{\La}$, $K^{\pm}_{\La}$ of escaping and non-escaping points respectively, are defined as
\[
I^{\pm}_{\Lambda} = \left\{ z = (x, y) \in \mbb C^2 : \Vert H_{n, \La}^{\pm}(z) \Vert \ra \infty \; \text{as} \; n \ra \infty \right\}
\]
and
\[
K^{\pm}_{\La} = \left\{ z = (x, y) \in \mbb C^2 : \text{the orbit} \; \big( H_{n, \La}^{\pm}(z)\big)_{n \ge 0} \; \text{is bounded in} \; \mbb C^2  \right\},
\]
and let $J^{\pm}_{\La} = \pa K^{\pm}_{\La}$.
\begin{prop}\label{P1}
For each $\La \in X$, the sequence $\{ G_{n, \La}^{\pm} \}$ converges locally uniformly on $\mbb C^2$ to a non-negative, locally H\"{o}lder continuous plurisubharmonic function $G_{\La}^{\pm}$ as $n \ra \infty$. The limit functions $G_{\La}^{\pm}$ satisfy the following invariance property: For $\la \in M$,
\[
G_{\La}^{\pm} \circ H^{\pm 1}_{\la} = d \; G_{\La'}^{\pm}
\]
where $\La' = \la \La$. The functions $G_{\La}^{\pm}$ are positive and pluriharmonic on $I^{\pm}_{\La} = \mbb C^2 \sm K^{\pm}_{\La}$, and vanish precisely on $K^{\pm}_{\La}$. Furthermore, the family $\{ G_{\La}^{\pm} \}$ is uniformly of logarithmic growth, i.e., there exists a $C > 0$ such that
\[
0 \le G_{\La}^{\pm}(z) \le \log^+ \vert z \vert + C
\]
for all $\La \in X$ and $z \in \mbb C^2$,  Finally, the correspondence $\La \mapsto G^{\pm}_{\La}$ is continuous.
\end{prop}

\no Let
\[
EG^{\pm}(z) = \int_M G_{\La}^{\pm}(z) \; d \ov m(\La)
\]
be the average Green functions for the family $\cal H$. Since $G^{\pm}_{\La}$ depends continuously on $\La$, the integral is well defined. The fact that the family $\{ G_{\La}^{\pm} \}$ is uniformly of logarithmic growth implies that
\[
0 \le EG^{\pm}(z) \le \log^+ \vert z \vert + O(1)
\]
near infinity.

\begin{prop}\label{P2}
$EG^{\pm}$ are locally H\"{o}lder continuous plurisubharmonic functions on $\mathbb C^2$. Moreover, $EG^{\pm}$ are strictly positive and pluriharmonic outside
\[
\cal K^{\pm} = \ov{\bigcup_{\La} K^{\pm}_{\La}}
\]
and vanish precisely on
\[
\cal K_0^{\pm} = \bigcap_{\La}K^{\pm}_{\La}.
\]
\end{prop}

\no In view of these observations, we may define the random stable and unstable Green currents as
\[
\mu^{\pm}_{\La} = \frac{1}{2\pi} dd^cG^{\pm}_{\La}
\]
for all $\La \in X$ and let
\[
\mu^{\pm} = \frac{1}{2\pi} dd^c(EG^{\pm})
\]
be the average stable and unstable Green currents respectively. These are well-defined positive closed $(1,1)$-currents of mass $1$ on $\mbb C^2$. Since the correspondence $\La \mapsto \mu^{\pm}_{\La}$ is continuous and $\mu^{\pm}_{\La}$ has mass $1$ for each $\La$, $\int_{X} \mu^{\pm}_{\La}$ defines a positive $(1,1)$-current on $\mbb C^2$ in the following way: for a test form $\phi$ on $\mbb C^2$,
\[
\left \langle \int_X \mu^{\pm}_{\La}, \phi  \right \rangle = \int_X \langle \mu^{\pm}_{\La}, \phi \rangle \; d \ov m(\La).
\]
It is useful to know how these objects behave when pushed forward or pulled back by an arbitrary H\'{e}non map from $\cal H$. The next proposition records this and also shows that the currents associated with the average Green functions agree with the corresponding average Green currents.

\begin{prop}\label{P3}
With $\mu^{\pm}_{\La}$ and $\mu^{\pm}$ as above,
\[
\mu^{\pm} = \int_X \mu^{\pm}_{\La}.
\]
Furthermore,
\[
(H_{\la}^{\pm 1})^{\ast} \mu_{\La}^{\pm} = d^{\pm 1} \mu_{\la\La}^{\pm}
\]
for any $\la \in M$ and $\La \in X$. Also,
\[
\int_M EG^{\pm} \circ H_{\la}^{\pm 1}(z) \; d m(\la) = d \; EG^{\pm}(z)
\]
and
\[
\int_M (H_{\la}^{\pm 1})^{\ast} \mu^{\pm} \; d m(\la) = d \; \mu^{\pm}.
\]
The support of $\mu_{\La}^{\pm}$ is $J_{\La}^{\pm}$ and the correspondences $\La \mapsto J_{\La}^{\pm}$ are lower semi-continuous. Finally,
\[
\text{supp}(\mu^{\pm}) = \ov{\bigcup_{\La \in X} J_{\La}^{\pm}}.
\]
\end{prop}

\no Since for each $\La \in X$, $\ov {K^+_{\La}}$ (the closure of $K^+_{\La}$ in $\mbb P^2$) equals $K^+_{\La} \cup I^+$, it follows that $\mu_{\La}^+$ admits an extension to $\mbb P^2$ as a positive closed $(1, 1)$-current, wherein the extension puts no mass on $I^+$. Likewise, as $EG^+$ has logarithmic growth near infinity, it follows that $\mu^+$ has bounded mass near the hyperplane at infinity $\{t = 0\}$ in $\mbb P^2$ and hence $\mu^+$ also extends to $\mbb P^2$ as a positive closed $(1,1)$-current. In both cases, the extensions will again be denoted by $\mu_{\La}^+$ and $\mu^+$.

\medskip

Let $\omega_{FS}$ be the standard Fubini-Study form on $\mbb P^2$ satisfying $\int_{\mbb P^2} \omega_{FS}^2 = 1$. By the $dd^c$-lemma, every positive closed $(1,1)$-current $S$ on $\mbb P^2$ can be written as $S = c \omega_{FS} + dd^c u$ for some integrable function $u$; here
\[
c = \int_{\mbb P^2} S \wedge \omega_{FS}
\]
is the mass of $S$. Such a $u$ is called a quasi-potential for $S$. Any pair of quasi-potentials $u_1, u_2$ must satisfy $dd^c(u_1 - u_2) = 0$ and hence the choice of a quasi-potential is unique modulo an additive constant. Let $\{S_n\}$ be a sequence of positive closed $(1,1)$-currents of mass $1$ on $\mbb P^2$ that admit a sequence $\{u_n\}$ of quasi-potentials that are uniformly bounded near $I^-$. Theorem 6.6 in Dinh--Sibony \cite{DS} shows that for a single H\'{e}non map $H$ of degree $d$, the sequence $d^{-n} (H^n)^{\ast}(S_n)$ converges exponentially fast to the corresponding Green current. It turns out that a stronger convergence property holds. In lieu of the iterates of a single map, it is possible to pull back the currents $S_n$ by $H_{n, \La}^+$ as in $(1.2)$ and the following theorem shows that this sequence converges to the corresponding random stable Green current exponentially fast.

\begin{thm}\label{thm1}
Let $\{S_n\}$ be a sequence of positive closed $(1,1)$-currents of mass $1$ on $\mbb P^2$. Assume that there exists $C > 0$ and a neighborhood $U$ of $I^-$ in $\mbb P^2$ such that each $S_n$ admits a quasi-potential $u_n$ that satisfies $\vert u_n \vert \le C$ on $U$. Then there exists a uniform constant $A > 0$ such that
\[
\vert \langle d^{-n} (H_{n, \La}^+)^{\ast}(S_n) - \mu_{\La}^+, \phi \rangle \vert \le And^{-n} \Vert \phi \Vert_{C^1}
\]
for all $\La \in X$ and for every $C^2$-smooth test form $\phi$ on $\mbb P^2$.
\end{thm}

\begin{rem}\label{rem1}
Once we prove the above theorem, following a similar string of arguments as in \cite{DS}, one can prove that for each $\Lambda\in X$, $K_\Lambda^+$ is rigid i.e. $\mu_\Lambda^+$ is the unique closed $(1,1)$-current of mass $1$ supported on the set $K_\Lambda^+$.
\end{rem}

Let $\cal S$ be the cone of positive closed $(1,1)$-currents on $\mbb P^2$. For $S \in \cal S$, define
\[
\Theta(S) = \frac{1}{d} \int_M H_{\la}^{\ast}S.
\]
It is evident that $\Theta(S)$ is also a positive closed $(1,1)$-current and thus $\Theta : \cal S \ra \cal S$ is well defined.

\begin{thm}\label{thm2}
Let $\{ S_n \}$ be a sequence of positive closed $(1,1)$-currents on $\mbb P^2$ of mass $1$. Assume that there exists $C > 0$  such that each $S_n$ admits a quasi-potential $u_n$ satisfying $\vert u_n \vert \le C$ on a fixed neighborhood $U$ of $I^-$ in $\mbb P^2$. Then $\Theta^n(S_n) \ra \mu^+$ as $n \ra \infty$ in the sense of currents.
\end{thm}

Following Dinh-Sibony \cite{DS}, recall that a set $A$ in a complex manifold $\Sigma$ is {\it rigid} if $A$ supports at most one non-zero positive closed $(1,1)$-current modulo multiplicative constants.

\begin{thm}\label{thm3}
For each $\La \in X$, the random Green current $\mu_{\La}^+$ is the unique positive closed $(1,1)$-current of mass $1$ supported on the set $\ov{K_{\La}^+}$. Hence, $\ov{K_{\La}^+}$ is rigid in $\mbb P^2$.
\end{thm}

\no While all this discussion extends and strengthens the work of Dinh--Sibony \cite{DS}, Bedford--Smillie (\cite{BS1, BS2, BS3}) to the random setting, it turns out that there are several consequences and benefits of such a generalization. A basic application that we had in mind was to understand the dynamical properties of skew products of H\'{e}non maps, which is a map $H : M \times \mbb C^2 \ra M \times \mbb C^2$ defined by
\begin{equation}\label{homeo}
H(\la, x, y) = (\si(\la), H_{\la}(x, y))
\end{equation}
where $M \subset \mbb C^k$ is compact, $\sigma : M \ra M$ is a homeomorphism and each $H_{\la}(x, y)$ is as in $(1.1)$. The $n$-fold iterates of $H^{\pm}$ are given by
\[
H^n(\la, x, y) = (\si^n(\la), H_{\si^{n-1}(\la)} \circ H_{\si^{n-2}(\la)} \circ \cdots \circ H_{\la}(x,y))
\]
and
\[
H^{-n}(\la, x, y) = (\si^{-n}(\la), H^{-1}_{\si^{-n}(\la)} \circ H^{-1}_{\si^{-(n-1)}(\la)} \circ \cdots \circ H^{-1}_{\si^{-1}(\la)}(x,y))
\]
respectively. Note that the second coordinate of $H^n(\la, x, y)$ is precisely of the form $H_{n, \La}^{+1}$ (as in $(1.2)$) with $\La = (\la, \si(\la), \si^2(\la), \ldots)$ and hence all considerations previously discussed in the random setup can be applied here. The first coordinate $\si^n(\la)$ does not come into play in a coarse understanding of the dynamics of $H$ since $M$ is compact. Not surprisingly, it does however influence the ergodic properties of $H$ as we shall see later. Though the study of skew products is of intrinsic interest, the choice of a map as in $(1.3)$ was primarily motivated by two examples. The first one, which is a prototype model, is a skew product over the circle $S^1 \subset \mbb C$ defined by
\begin{equation}
H_{\al}(\la, x, y) = (e^{2\pi i \alpha} \la, H_{\la}(x, y))
\end{equation}
where $\la \in S^1$ and $\alpha$ is a fixed irrational. Since much is known about each fiber map $H_{\la}(x, y)$, it seemed plausible to understanding this family of maps over an irrational rotation by attempting to lift all known considerations to $S^1 \times \mbb C^2$. Second, let us recall the the Fornaess--Wu \cite{FW} classification of polynomial automorphisms of $\mbb C^3$ of degree at most $2$. This classification has $7$ families and several of them are skew products of H\'{e}non maps for suitable choices of the parameters that are involved. It must also be mentioned that Coman--Fornaess \cite{CF} have studied the dynamics of these classes of automorphisms by treating them as self-maps of $\mbb C^3$; the view point adopted here (and in \cite{PV}) is to think of them as skew products with the intention of studying their ergodic properties. In \cite{PV}, the hypotheses on $\si$ were rather general -- in fact, $\si$ was assumed to be only continuous. A moment's reflection shows that the global map $H = H(\la, x, y)$ (as in $(1.3)$) is no longer invertible then. In this situation, the lower bound on the topological entropy of $H$ that was obtained in \cite{PV} did not reflect the contribution of $\si$. However, if we assume that $\si : M \ra M$ is a homeomorphism (which fits in well with several sub-cases of the Fornaess--Wu classification), it is possible to say much more -- among other things, we obtain better lower bounds for the topological entropy of $H$ as well as its largest Lyapunov exponent.

\medskip

 In what follows then, we will work with the global map $H = H(\la, x, y)$ as in $(1.3)$ and $\si : M \ra M$ will be assumed to be a {\it homeomorphism}. The first step is to make all this discussion precise.

\medskip

 For each $\la \in M$, the sets $I^{\pm}_{\la}$, $K^{\pm}_{\la}$ of non-escaping and escaping points respectively are defined as
\[
I^{\pm}_{\lambda} = \left\{ z = (x, y) \in \mbb C^2 : \Vert H_{\la}^{\pm n}(z) \Vert \ra \infty \; \text{as} \; n \ra \infty \right\}
\]
and
\[
K^{\pm}_{\la} = \left\{ z = (x, y) \in \mbb C^2 : \text{the orbit} \; \big( H_{\la}^{\pm n}(z)\big)_{n \ge 0} \; \text{is bounded in} \; \mbb C^2  \right\}.
\]
For each $\la \in M$, it can be seen that $H_{\la}(K^{\pm}_{\la}) = K^{\pm}_{\si(\la)}$ and $H_{\la}(I^{\pm}_{\la}) = I^{\pm}_{\si(\la)}$. Let $K_{\la} = K^+_{\la} \cap K^-_{\la}$, $J^{\pm}_{\la} = \pa K^{\pm}_{\la}$ and $J_{\la} = J^+_{\la} \cap J^-_{\la}$. Thus $H_{\la}(K_{\la}) = K_{\si(\la)}$ and $H_{\la}(J_{\la}) = J_{\si(\la)}$.

\medskip

The sets $I^{\pm}_{\la}, K^{\pm}_{\la}$ can be put together as
\[
I^{\pm} = \bigcup_{\la \in M} \{ \la \} \times I^{\pm}_{\la} \; \text{and} \; K^{\pm} = \bigcup_{\la \in M} \{ \la \} \times K^{\pm}_{\la}
\]
and since $M$ is compact, these can be interpreted as
\[
I^{\pm} = \left\{ (\la, x, y) \in M \times \mbb C^2 : \Vert H^{\pm n}(\la, x, y) \Vert \ra \infty \; \text{as} \; n \ra \infty  \right\}
\]
and
\[
K^{\pm} = \left\{ (\la, x, y) \in M \times \mbb C^2 : \text{the orbit} \; \big( H^{\pm n}(\la, x, y)\big)_{n \ge 0} \; \text{is bounded in} \; \mbb C^2  \right\}.
\]
respectively. Both $I^{\pm}, K^{\pm}$ are completely invariant under $H$.

\medskip

For $n \ge 0$ and $(x, y) \in \mbb C^2$, let
\[
G^{\pm}_{n, \la}(x, y) = \frac{1}{d^n} \log^+ \Vert \pi_2 \circ H^{\pm n}(\la, x, y) \Vert,
\]
where $\pi_2$ is the projection on the second factor. For a fixed $\la \in M$, Proposition $1.1$ shows that the sequence $G^{\pm}_{n, \la}(x, y)$ converges uniformly on compact subsets of $\mbb C^2$ to a continuous function $G_{\la}^{\pm}$ as $n \ra \infty$ which satisfies
\[
G^{\pm}_{\si(\la)} \circ H_{\la} = d^{\pm 1} G^{\pm}_{\la}
\]
on $\mbb C^2$. This invariance property agrees with what is known for a single H\'{e}non map (which happens when $M$ is a single point). The functions $G^{\pm}_{\la}$ are positive pluriharmonic on $I^{\pm}_{\la} = \mbb C^2 \sm K^{\pm}_{\la}$, plurisubharmonic on $\mbb C^2$ and vanish precisely on $K^{\pm}_{\la}$. The correspondences $\la \mapsto G^{\pm}_{\la}$ are continuous and $G^{\pm}_{\la}$ are locally uniformly H\"{o}lder continuous on $\mbb C^2$. Thus $\mu_{\la}^{\pm} = 1/2 \pi dd^c G^{\pm}_{\la}$ are positive closed $(1,1)$-currents on $\mbb C^2$ of mass $1$ and $\mu_{\la} = \mu^+_{\la} \wedge \mu^-_{\la}$ is a family of probability measures on $\mbb C^2$.

\medskip

It can also be checked, along the lines of Proposition $1.3$, that
\[
(H_{\la})^{\ast}  \mu^{\pm}_{\si(\la)} = d^{\pm 1} \mu^{\pm}_{\la}, \; (H_{\la})_{\ast} \mu^{\pm}_{\la} = d^{\mp 1} \mu^{\pm}_{\si(\la)} \;\text{and} \; (H_{\la})^{\ast} \mu_{\si(\la)} = \mu_{\la}.
\]
Furthermore, the support of $\mu^{\pm}_{\la}$ equals $J^{\pm}_{\la}$ and $\la \mapsto J^{\pm}_{\la}$ is lower semi-continuous. Finally, for each $\la \in M$, the pluricomplex Green function of the compact set $K_{\la}$ is $\max \{ G^+_{\la}, G^-_{\la} \}$, $\mu_{\la}$ is the equilibrium measure of $K_{\la}$ and the support of $\mu_{\la}$ is contained in $J_{\la}$ (see \cite{BS1}, \cite{PV}).

\medskip

Let $\mu'$ be an invariant probability measure for $\si$ on $M$. Then
\begin{equation}
\langle \mu, \varphi \rangle = \int_M \left( \int_{\{\la\} \times \mbb C^2} \varphi \; \mu_{\la}\right) \mu'(\la)
\end{equation}
defines a measure on $M \times \mbb C^2$ by describing its action on continuous functions $\varphi$ on $M \times \mbb C^2$. Since $(H_{\la})^{\ast} \mu_{\si(\la)} = \mu_{\la}$, it follows that $H^{\ast} \mu = \mu$ on $M \times \mbb C^2$. Thus $\mu$ is an invariant probability measure for $H$. 
 
\begin{thm}\label{thm4}
If $\si$ is mixing for $\mu'$, then $H$ is mixing for $\mu$, i.e.,
\[
\lim_{n \ra \infty} \int (H^{n \ast} \varphi) \psi \; d \mu = \int \varphi\; d\mu \int \psi \; d \mu
\]
for all continuous functions $\varphi, \psi$ with compact support on $M \times \mbb C^2$.
\end{thm}

On the other hand, there is a related notion of mixing in the random setup. Following \cite{P}, for each $n \ge 0$, let $\{ \nu_n \}$ be a sequence of probability measures on a space, say $Y$ and $T_n : Y \ra Y$ a family of transformations that preserve the sequence $\{ \nu_n\}$, i.e., $T^{\ast}_n(\nu_{n+1}) = \nu_n$ for each $n \ge 0$. We say that the measure preserving sequence of transformations $\{T_n\}$ is randomly mixing for the sequence $\{ \nu_n\}$ if
\[
\int (T^{\ast}_n \varphi) \psi \; d \nu_0 - \int \varphi \; d \nu_n \int \psi \; d \nu_0 \ra 0
\]
as $n \ra \infty$. Note that for each $\la  \in M$, the sequence of H\'{e}non maps $\{ H_{\si^n(\la)} : n \ge 0\}$ is measure preserving for the sequence of probability measures  $\{ \mu_{\si^n(\la)} : n \ge 0\}$ on $\mbb C^2$.

\begin{prop}\label{P4}
For each $\la \in M$, the sequence of H\'{e}non maps $\{ H_{\si^n(\la)} : n \ge 0\}$ is randomly mixing for the sequence of probability measures $\{ \mu_{\si^n(\la)} : n \ge 0\}$.
\end{prop}

Note that Theorem \ref{thm4} assures the existence of Lyapunov exponents of $H$ with respect to $\mu$ for a good choice of $\mu'$. Thus if we consider $M$ to be compact set in $\mathbb{C}^k$, there are at most $(k+2)$ Lyapunov exponents of $H$. Let $\lambda_1$ be the largest one, then 
we have the following:
\begin{thm}\label{thm6}
$\lambda_1 \geq \max \{\log d,\lambda_\sigma\}$ where $\lambda_\sigma$ is the largest Lyapunov exponent of $\sigma$ with respect to $\mu'$.
\end{thm}

The definition of $\mu$ shows that its support is contained in
\[
J = \ov{\bigcup_{\la \in M}( \{ \la \} \times J_{\la})}
\]
which is a compact set in $M \times \mbb C^2$. We may therefore consider $H$ as a self map of supp$(\mu)$ with invariant measure $\mu$. Let $h(H ; \mu)$ and $h(\si ; \mu')$ be the measure theoretic entropies  of $H$ with respect to $\mu$ and of $\si$ with respect to the invariant measure $\mu'$ respectively. By adapting the arguments of Bedford--Smillie \cite{BS3} and appealing to the Abramov--Rohlin theorem \cite{AR}, it is possible to obtain lower bounds for the entropies of $H$ that reflect the contribution of $\si$.

\begin{thm}\label{thm5}
With $H$, $\mu$ as above,
\[
h(H; \mu) \geq h(\si; \mu') + \log d.
\]
In particular, the topological entropies of $H, \si$ are related by  $h_{top}(H) \geq h_{top}(\si) + \log d$. 

\end{thm}
This strengthens the weaker lower bound of $\log d$ obtained in \cite{PV}. It is also possible to adapt Smillie's arguments \cite{Sm} to obtain the aforementioned lower bound for the topological entropy  $h_{top}(H)$ directly as well.


\section{Random iterations of H\'{e}non maps}

For $R>0$, let
\begin{align*}
V_R^+ &= \big\{ (x, y) \in \mbb C^2 : \vert y \vert > \vert x \vert, \vert y \vert > R \big\},\\
V_R^- &= \big\{ (x, y) \in \mbb C^2 : \vert y \vert < \vert x \vert, \vert x \vert > R \big\}\text{ and }\\
V_R   &= \big\{ (x, y) \in \mbb C^2 : \vert x \vert, \vert y \vert \le R\}.
\end{align*}
The compactness of $M$ implies that for a sufficiently large $R>0$,
$$
H_\lambda(V_R^+)\subset V_R^+, \ \ H_\lambda(V_R^+\cup V_R)\subset V_R^+\cup V_R
$$
and 
$$
H_\lambda^{-1}(V_R^-)\subset V_R^-, \ \ H_\lambda^{-1}(V_R^-\cup V_R)\subset V_R^-\cup V_R
$$
for all $\lambda\in M$.
As in Lemma $2.1$ in \cite{PV}, it can be shown that  
$$
I_\Lambda^{\pm}=\mathbb{C}^2\setminus K_\Lambda^{\pm}=\bigcup_{n=0}^\infty {(H_{n,\Lambda}^\pm)}^{-1}(V_R^{\pm})
$$
for each $\Lambda\in X$.

\subsection{Proof of Proposition \ref{P1}}
\begin{proof}
Since $I^-$ is an attracting fixed point for each $H_\lambda$ and the correspondence $\lambda\mapsto H_\lambda$ is continuous, one can choose a suitable neighborhood of $I^-$ in $\mathbb{P}^2$, say $U$, such that $H_\lambda(Y)\subset \subset Y$ for all $\lambda\in M$ where $Y=\mathbb{P}^2\setminus \overline{U}$. Now start with a compact set $K$ in $\mathbb{C}^2$. Let $R$ be so large that $K\subset V_R\cup V_R^+$ and $H_\lambda(V_R\cup V_R^+)\subset V_R \cup V_R^+$ for all $\lambda\in M$. Define 
$$
v_\lambda(z)=\frac{1}{d}\log^+ \lVert H_\lambda(z)\rVert-\log^+\lVert z \rVert
$$
for $z\in V_R\cup V_R^+$. The proof of Proposition $1.1$ in \cite{PV} provides a constant $K>0$ such that $\lvert v_\lambda(z) \rvert < K$ for all $z\in V_R^+ \cup V_R $ and $\lambda\in M$. Consider $(Y,d_{FS})$ where $d_{FS}$ is the Fubini-Study metric on $Y$. Thus $Y$ is a metric space of finite diameter. Note that  each $v_\lambda$ and $H_\lambda$ are Lipschitz on $Y$ with  uniform Lipschitz coefficients $A$ and $L$ respectively. In particular, it is possible to choose $L$ to be $\sup_{\lambda\in M} L_\lambda$ where each $L_\lambda=\sup_{\mathbb{P}^2\setminus U}\lVert DH_\lambda \rVert$.   \\

\noindent
Now 
\begin{equation*}
G_{n+1,\Lambda}^+(z)-G_{n,\Lambda}^+(z) = \frac{1}{d^{n+1}} \log^+ \big \lVert H_{n+1,\Lambda}^+(z)\big \rVert-\frac{1}{d^{n}} \log^+ \big\lVert H_{n,\Lambda}^+(z)\big \rVert 
= \frac{1}{d^n} \big ( v_{\lambda_{n+1}}\circ H_{n,\Lambda}^+(z)\big ).
\end{equation*}
Thus 
\begin{equation}
\big \lvert G_{n+1,\Lambda}^+(z)-G_{n,\Lambda}^+(z)\big \rvert=d^{-n}\big\lvert  v_{\lambda_{n+1}}\circ H_{n,\Lambda}^+(z) \big\rvert \leq Kd^{-n} \label{R2}
\end{equation}
on $V_R\cup V_R^+$ for all $n\geq 1$ and for all $\Lambda\in X$. Therefore $G_{n,\Lambda}^+$ converges uniformly to a continuous function $G_\Lambda^+$ on any compact subset of $\mathbb{C}^2$. Also note that 
\[
d G_{n+1,\Lambda'}^+=G_{n,\Lambda}^+\circ H_\lambda
\]
for all $n\geq 1$, which yields $d G_{\Lambda'}^+=G_\Lambda^+ \circ H_\lambda$. That $G_\Lambda^{+}$ is a strictly positive pluriharmonic function on $\mathbb{C}^2\setminus K_\Lambda^{+}$ and vanishes precisely on $K_\Lambda^{+}$ follows by using arguments similar to those in Proposition $1.1$ in \cite{PV}.\\
\indent
Note that 
\[
G_\Lambda^+(z)=\log^+\lVert z \rVert+\sum_{n\geq 0} d^{-n}\big(v_{\lambda_{n+1}} \circ H_{n,\Lambda}^+(z)\big) 
\]
for all $\Lambda\in X$. Since $(Y,d_{FS})$ has finite diameter, it is sufficient to work with $a,b \in V_R\cup V_R^+$ with $d_{FS}(a,b)<<1$. Now
\begin{align}\label{R3}
&\Big\lvert\sum_{n\geq 0}d^{-n}v_{\lambda_{n+1}}\circ H_{n,\Lambda}^+(a)- \sum_{n\geq 0}d^{-n}v_{\lambda_{n+1}}\circ H_{n,\Lambda}^+(b) \Big\rvert \nonumber\\
&\leq \sum_{0\leq n \leq N-1} d^{-n}\big\lvert v_{\lambda_{n+1}}\circ H_{n,\Lambda}^+(a)- v_{\lambda_{n+1}}\circ H_{n,\Lambda}^+(b)\big\rvert 
+ \sum_{n \geq N} d^{-n}\big\lvert v_{\lambda_{n+1}}\circ H_{n,\Lambda}^+(a)- v_{\lambda_{n+1}}\circ H_{n,\Lambda}^+(b)\big\rvert \nonumber\\
&\leq  A \sum_{0\leq n \leq N-1} d^{-n} \big\lvert H_{n,\Lambda}^+(a)-H_{n,\Lambda}^+(b)\big\rvert + 2K \sum_{n\geq N} d^{-n} 
\leq A d_{FS}(a,b) \sum_{0\leq n \leq N-1} d^{-n}L^n + K'd^{-N}\nonumber\\
\end{align}
where $K'={2Kd}/{(d-1)}$. If $L<d$, the last sum in (\ref{R3}) is at most $ d_{FS}(a,b)+ d^{-N}$ upto a constant. So for a given $0<\beta<1$, if we choose 
$$
N> -\frac{\beta \log d_{FS}(a,b)}{\log d},
$$
then the sum in (\ref{R3}) is $\lesssim {d_{FS}(a,b)}^\beta$. Thus in this case, $G_\Lambda^+$ is locally H\"{o}lder continuous in $\mathbb{C}^2$ for any $0<\beta <1$. For $L\geq d$, the last sum in (\ref{R3}) is at most $d_{FS}(a,b){(Ld^{-1})}^N+d^{-N}$. Now if we choose 
$$
-\frac{\log d_{FS}(a,b)}{\log L}\leq N <-\frac{\log d_{FS}(a,b)}{\log L}+1,
$$
the sum in (\ref{R3}) is $ \lesssim {d_{FS}(a,b)}^{{\log d}/{\log L}}$ since $d^{-N}< {d_{FS}(a,b)}^{{\log d}/{\log L}}$. Therefore $G_\Lambda^+$ is locally H\"{o}lder continuous in $\mathbb{C}^2$ for any $0<\beta <{\log d}/{\log L}$. So for each $\Lambda\in X$, $G_\Lambda^+$ is $\beta$-H\"{o}lder continuous for any $\beta$ such that $0<\beta< \min\{1,{\log d}/{\log L}\}$. A similar argument shows that $G_\Lambda^-$ is H\"{o}lder continuous. 

\medskip

That the correspondence $\Lambda \mapsto G_\Lambda^{+}$ is continuous follows by using arguments similar to those in Proposition $1.1$ in \cite{PV}.

\end{proof}

\subsubsection{Proof of Proposition \ref{P2}}  
\begin{proof}
Pick a $z\in \mathbb{C}^2$ and let $\Delta$ be a small disk centered at $z$. Then 
\begin{multline*}
EG^+(z)=\int_X G_\Lambda^+ (z)
\leq \int_X \frac{1}{\lvert \partial \Delta \rvert} \int_{\partial \Delta} G_\Lambda^+ (\xi) 
= \frac {1}{\lvert\partial \Delta \rvert} \int_{\partial \Delta} \int_X G_\Lambda^+(\xi)
= \frac{1}{\lvert\partial \Delta \rvert} \int_{\partial \Delta} EG^+(\xi).
\end{multline*}
The first inequality follows due to the fact that $G_\Lambda^+$ is subharmonic in $\mathbb{C}^2$ for each $\Lambda\in X$ and the next equality is obtained by applying Fubini's theorem to the continuous function $G^+: X \times \mathbb{C}^2 \ra \mathbb{R}$ defined by $(\Lambda, z) \mapsto G_\Lambda^+(z)$. This shows that $EG^+$ is plurisubharmonic. That for a fixed compact set  in $\mathbb{C}^2$, a fixed H\"{o}lder coefficient and a fixed H\"{o}lder exponent work for each $G_\Lambda^+$ implies the local H\"{o}lder continuity and in particular, the continuity of the average Green function $EG^+$ in $\mathbb{C}^2$. A similar statement is valid for $EG^{-}$.  

\medskip

\no
For $z\in \bigcap_{\Lambda\in X} K_\Lambda^+$, it is evident that $EG^+(z)=0$. To see the converse, let 
\[
EG^+(z)=\int_{\Lambda\in X} G_\Lambda^+(z)=0
\] 
for a $z\in \mathbb{C}^2$. Since the correspondence $\Lambda\mapsto G_\Lambda^+(z)$ is continuous for a fixed $z\in \mathbb{C}^2$ and $G_\Lambda^+$ is positive on $\mathbb{C}^2$ for all $\Lambda\in X$, we have $G_\Lambda^+(z)=0$ for all $\Lambda\in X$. Hence $z\in \bigcap_{\Lambda} K_\Lambda^+$. Further, it is not difficult to observe that $EG^{\pm}$ is pluriharmonic outside $\overline{\bigcup_{\Lambda}K_\Lambda^\pm}$.
\end{proof}

\subsection{Proof of Proposition \ref{P3}} 
\begin{proof}
For a test form $\varphi$ in $\mathbb{C}^2$, Note that
\begin{equation*}
\langle \mu^\pm,\varphi\rangle=\langle dd^c{EG}^{\pm},  \varphi\rangle
= \int_X\langle G_\Lambda^{\pm}, dd^c \varphi\rangle 
=\int_X \langle dd^c G_\Lambda^{\pm},\varphi\rangle
= \int_X \langle \mu_\Lambda^{\pm},\varphi\rangle
=\Big \langle \int_X \mu_\Lambda^{\pm},\varphi \Big\rangle.
\end{equation*}

Thus $\mu^{\pm}=\int_X \mu_\Lambda^{\pm}$.

\medskip

From Proposition \ref{P1}, it follows that  $G_\Lambda^{\pm}\circ H_\lambda^{\pm 1}=d G_{\lambda\Lambda}^{\pm}$
which in turn gives 
$$
{(H_\lambda^{\pm 1})}^*(dd^c G_\Lambda^\pm)= d (dd^c G_{\lambda\Lambda}^\pm).
$$
Hence ${(H_\lambda^{\pm 1})}^* \mu_\Lambda^{\pm}=d \mu_{\lambda\Lambda}^\pm$ for all $\lambda\in M$ and for all $\Lambda\in X$.

\medskip

That for a $z\in \mathbb{C}^2$, 
$$
\int_{M} {EG}^{\pm}\circ H_\lambda^{\pm}(z)d\nu(\lambda)= d {(EG)}^{\pm}(z)
$$
follows from the fact that $G_\Lambda^{\pm}\circ H_\lambda^{\pm 1}=d G_{\lambda \Lambda}^{\pm}$ for all $\lambda\in M$ and for all $\Lambda\in X$ .

\medskip

To prove the next assertion observe that
\begin{align*}
&\Big\langle \int_{M} {(H_\lambda^{\pm 1})}^* \mu^{\pm},\varphi \Big\rangle=\int_{M} \langle dd^c ({EG}^{\pm}\circ H_\lambda^{\pm 1 }),\varphi \rangle
=\int_{M}\langle {EG}^{\pm}\circ H_\lambda^{\pm 1 }, dd^c \varphi \rangle\\
&=\Big \langle \int_{M} {EG}^{\pm}\circ  H_\lambda^{\pm 1 },dd^c \varphi\Big\rangle
= \langle d {(EG)}^{\pm},dd^c \varphi\rangle
= d \langle dd^c ({EG}^{\pm}),\varphi \rangle
= d \langle \mu^{\pm},\varphi\rangle.
\end{align*}
Therefore $\int_{M} {(H_\lambda^{\pm 1})}^* \mu^{\pm}=d \mu^{\pm}$.

\medskip

That the support of $\mu_\Lambda^\pm$ is $J_\Lambda^\pm$ and the correspondence $\Lambda\mapsto J_\Lambda^\pm$ is lower semi-continuous can be shown as in Proposition $1.2$ in \cite{PV}.

\medskip

Since $\mu^\pm=\int_{X} \mu_\Lambda^\pm$, we have ${ \rm supp} (\mu^{\pm})\subset \overline{\bigcup_{\Lambda\in X} J_\Lambda^\pm }$. To prove the other inclusion, first note that $\mu_\Lambda^{\pm}$ vary continuously in $\Lambda$. Therefore, if some $\mu_{\Lambda_0}^{\pm}$ has positive mass in some open set $\Omega\subset\mathbb{C}^2$, then there exists a neighborhood $U_{\Lambda_0}\subset X$ of $\Lambda_0$ such that $\mu_\Lambda^\pm$ has nonzero mass in $\Omega$ for all $\Lambda\in U_{\Lambda_0}$ and thus $\mu^{\pm}$ also has nonzero mass in $\Omega$. This completes the proof. 
\end{proof}

\section*{Convergence to the random stable Green current}

\subsection{Proof of Theorem \ref{thm1}}

\begin{proof}
Since the correspondence $\lambda\mapsto H_\lambda$ is continuous, we can assume that $U$ satisfies that $H_\lambda(U)\subset \subset U$ for all $\lambda\in M$. Let $\varphi$ be a $(1,1)$-form in $\mathbb{P}^2$. Since $\varphi$ is of class $C^2$, $dd^c \varphi$ is a continuous form of maximal degree. Thus we can assume that the signed measure given by $dd^c \varphi$ has no mass on a set of volume zero and in particular, on the hyperplane at infinity. Hence multiplying $\varphi$ with a suitable constant, we can assume 
${\lVert dd^c \varphi \rVert}_\infty \leq 1$. Thus $\gamma=dd^c \varphi$ is a complex measure in $\mathbb{C}^2$ with mass less than or equal to $1$.   
 Define $\gamma_{n}={(H_{\lambda_n}\circ \cdots\circ H_{\lambda_1})}_* (\gamma)$ for all $n\geq 1$. Note that $\gamma_n$ has the same mass as $\gamma$ for all $n\geq 1$ since $H_\lambda$ is automorphism of $\mathbb{C}^2$ for all $\lambda\in M$. Let ${\gamma_n}'$ and ${\gamma_n}''$ be the restrictions of $\gamma_n$ to $\mathbb{P}^2 \setminus U$ and $U$ respectively. Clearly, $\lVert {\gamma_n}' \rVert \leq 1$ on $\mathbb{P}^2\setminus U$ and $ \lVert{\gamma_n}''\rVert \leq 1$ on $U$. Observe that due to the choice of $U$, $H_\lambda^{-1}$ defines a map from $\mathbb{P}^2\setminus U$ to $\mathbb{P}^2\setminus U$ with $C^1$ norm bounded by some $L>0$ for all $\lambda\in M$. So the $C^1$ norm of 
 $(H_{\lambda_1}^{-1}\circ\cdots\circ H_{\lambda_n}^{-1})$ is bounded by $L^n$ on $\mathbb{P}^2\setminus U$ for all $n\geq 1$ and consequently, we have  ${\lVert\gamma_n' \rVert}_{\infty}\leq L^{4n}$ for all $n\geq 1$.\\
 
\medskip 

For each $n\geq 1$, define:
\begin{equation*}
 q_{\Lambda_n}^+ := \frac{1}{2\pi}G_{\Lambda_n}^+ - \frac{1}{2\pi} \log {(1+ {\lVert z \rVert}^2)}^{\frac{1}{2}} .
\end{equation*}
 Note that $ q_{\Lambda_n}^+$ is a quasi-potential of $ \mu_{\Lambda_n}^+$ for all $n\geq 1$. Let $v_{\Lambda_n} := u_n -  q_{\Lambda_{n+1}}^+$.
By Proposition $1.1$ in \cite{PV} (See $(2.15)$ therein), one can show that $q_{\Lambda_n}^+$ is bounded on $U$ by a fixed constant for all $n\geq 1$. This implies that there exists $A_1>0$ such that $\lvert v_{\Lambda_n}\rvert \leq A_1$ on $U$ for all $n\geq 1$. Now observe that
\begin{align}
&\big\langle  d^{-n} {(H_{\lambda_n}\circ \cdots\circ H_{\lambda_1} )}^* ({S_n})-\mu_\Lambda^+, \varphi \big\rangle \nonumber\\
&=\big\langle d^{-n} {(H_{\lambda_n}\circ \cdots\circ H_{\lambda_1} )}^* ({S_n})- d^{-n} {(H_{\lambda_n}\circ \cdots\circ H_{\lambda_1} )}^* (\mu_{\Lambda_{n+1}}^+),\varphi \big\rangle \nonumber \\
&=d^{-n} \big\langle {(H_{\lambda_n}\circ \cdots\circ H_{\lambda_1} )}^* (dd^c (v_{\Lambda_n})), \varphi \big\rangle 
=d^{-n} \big\langle v_{\Lambda_n}, {(H_{\lambda_n}\circ \cdots\circ H_{\lambda_1} )}_*(dd^c \varphi) \big\rangle \nonumber\\
&=d^{-n} \big\langle v_{\Lambda_n}, \gamma_n  \big\rangle  
= d^{-n} \big\langle \gamma_n',  v_{\Lambda_n} \big\rangle + d^{-n} \big\langle \gamma_n'', v_{\Lambda_n} \big\rangle 
\label{R4}
\end{align}
\indent
Consider the family $\{v_{\Lambda_n}\}_{n\geq 1}$. Note that $dd^c (v_{\Lambda_n})= {S_n}- \mu_{\Lambda_{n+1}}^+$. This implies that 
${\lVert dd^c (v_{\Lambda_n})\rVert}_*$ is uniformly bounded by a fixed constant for all $n\geq 1$ since ${S_n}$ and $\mu_{\Lambda_n}^+$ both have mass $1$ for all $n\geq 1$. Here 
$$
{\lVert dd^c (v_{\Lambda_n}) \rVert}_*:=\inf(\lVert S_n^+ \rVert-\lVert S_n^-\rVert)
$$
where the infimum is taken over all positive closed $(1,1)$ currents $S_n^{\pm}$ such that $v_{\Lambda_n}=S_n^+-S_n^-$. Hence by Lemma $3.11$ in \cite{DS}, it follows that $\{v_{\Lambda_n}\}_{n\geq 1}$ is a bounded subset in DSH$(\mathbb{P}^2)$. Since $ \lVert\gamma_n'\rVert \leq 1$ and ${\lVert\gamma_n'\rVert}_\infty \leq L^{4n}$, Corollary {3.13} in \cite{DS} shows that
\begin{equation}
\big \lvert d^{-n} \big\langle \gamma_n', v_{\Lambda_n}\big \rangle \big \rvert \leq d^{-n} c (1+\log^+  {L^{4n}}) \lesssim nd^{-n}.
\end{equation} 
Now the second term in (\ref{R4}) is $O(d^{-n})$ since $\lVert \gamma_n'' \rVert \leq 1$ and $\lvert v_{\Lambda_n}\rvert \leq A_1$ on $U$. This estimate along with (\ref{R4}) gives 
\begin{equation}
\big\lvert \big\langle  d^{-n} {(H_{\lambda_n}\circ \cdots\circ H_{\lambda_1} )}^* ({S_n})-\mu_\Lambda^+, \phi \big\rangle \big\rvert 
\leq And^{-n} {\lVert \phi \rVert}_{C^1}.\label{uni}
\end{equation}
\end{proof}

\begin{rem}
It is clear from (\ref{uni}) and Lemma $3.11$ in \cite{DS} that the constant $A$ does not depend on the $\Lambda$ that we start with. In particular, it shows that for given currents ${\{S_n\}}_{n\geq 1}$ as prescribed before $d^{-n} {(H_{\lambda_n}\circ \cdots\circ H_{\lambda_1} )}^*(S_n)$ converges uniformly to $\mu_\Lambda^+$  for all $\Lambda\in X$  in the weak sense of currents.
\end{rem}

\subsection{Proof of Theorem \ref{thm2}}
\begin{proof}
First note that 
\begin{equation*}
\Theta^n(S_n)=\int_{M^n} {\frac{{(H_{n,\Lambda}^+)}^*(S_n)}{d^n}}.
\end{equation*}
For a test form $\varphi$ on $\mathbb{P}^2$, we have
\begin{eqnarray}
\big\langle \Theta^n(S_n),\varphi \big\rangle &=& \int_{M^n} \Big\langle \frac{{(H_{n,\Lambda}^+)}^*(S_n)}{d^n},\varphi \Big\rangle \nonumber 
= \int_{X} \Big\langle \frac{{(H_{n,\Lambda})}^*(S_n)}{d^n},\varphi \Big\rangle.
\end{eqnarray}
For each $n\geq 1$, we define $F_n: X \ra \mathbb{C}$ as follows:
$$
F_n(\Lambda):=  \Big\langle \frac{{(H_{n,\Lambda})}^*(S_n)}{d^n},\varphi \Big\rangle.
$$ 
By Theorem \ref{thm1}, 
$$
F_n(\Lambda) \ra F(\Lambda)=\langle \mu_\Lambda^+,\varphi\rangle
$$
for all $\Lambda\in X$. Note that each $F_n$ and $F$ are integrable. Also we have $\lvert F_n(\Lambda)\rvert \lesssim {\lVert \varphi \rVert}_\infty$ for all $n\geq 1$ and for all $\Lambda\in X$. By the dominated convergence theorem 
\begin{eqnarray*}
\big\langle\Theta^n(S_n),\varphi \big\rangle =\int_X F_n &\ra& \int_X F 
= \int_X \big\langle \mu_\Lambda^+,\varphi \big\rangle
= \langle \mu^+,\varphi \rangle.
\end{eqnarray*}
This completes the proof.
\end{proof}

\subsection{Proof of Theorem \ref{thm3}}
\begin{proof}
Fix $\Lambda\in X$ and let $S$ be a positive closed $(1,1)$-current of mass $1$ with support in $\overline{K_\Lambda^+}$. We show that $S=\mu_\Lambda^+$.  For each $n\geq 1$, define 
$$
S_{n,\Lambda}=d^n (H_{\lambda_n}\circ \cdots\circ H_{\lambda_1})_* S
$$
on $\mathbb{C}^2$. Note that each $S_{n,\Lambda}$ is a positive closed $(1,1)$-current on $\mathbb{C}^2$ with support in $K_{\Lambda_n}^+$. Therefore it can be extended through the hyperplane by $0$ as a positive closed $(1,1)$-current on $\mathbb{P}^2$. Now since 
$$
{(H_{\lambda_n}\circ \cdots\circ H_{\lambda_1})}^*{(H_{\lambda_n}\circ \cdots\circ H_{\lambda_1})}_* (S_{\Lambda,n})=S
$$  
on $\mathbb{C}^2$, we have
\begin{equation*}
d^{-n}{(H_{\lambda_n}\circ \cdots\circ H_{\lambda_1})}^* (S_{n,\Lambda})=S
\end{equation*}
on $\mathbb{C}^2$ and thus
\begin{equation}
d^{-n}{(H_{\lambda_n}\circ \cdots\circ H_{\lambda_1})}^* (S_{n,\Lambda})=S \label{R5}
\end{equation}
on $\mathbb{P}^2$. As a current on $\mathbb{P}^2$ each $S_{n,\Lambda}$ vanishes in a neighborhood of $I^-$  and is of mass $1$. Since for each $n\geq 1$, as a current on $\mathbb{C}^2$, $S_{\Lambda,n}$ has support in $K_{\Lambda_n}^+$, we have ${\rm{supp}}(S_{n,\Lambda})\cap \overline{V_R^+}=\phi$ for sufficiently large $R>0$. By Proposition $8.3.6$ in \cite{MNTU}, for each $n\geq 1$ there exists $u_{n,\Lambda}$ such that 
$$
S_{n,\Lambda}=c_{n,\Lambda} dd^c (u_{n,\Lambda})
$$ 
where $u_{n,\Lambda}(x,y)-\log\lvert y \rvert$ is a bounded pluriharmonic function  on $\overline{V}_R^+$ and $c_{n,\Lambda}>0$. Hence ${(S_{n,\Lambda})}_{n\geq 1}$ satisfies the required hypothesis of Theorem \ref{thm1} and thus we get  
$$
d^{-n}{(H_{\lambda_n}\circ \cdots\circ H_{\lambda_1})}^* (S_{n,\Lambda})\ra \mu_\Lambda^+
$$
in the sense of currents as $n\ra \infty$. Therefore by Theorem \ref{R5}, we have $S=\mu_\Lambda^+$.
\end{proof}

\begin{cor}\label{Cor1}
For each $\Lambda\in X$, the random Julia set $J_\Lambda^+=\partial K_\Lambda^+$ is rigid.
\end{cor}


\section{Skew products of H\'{e}non maps: Mixing properties}
Bedford-Smillie have shown that for a single H\'{e}non map $H$, $T$ a positive closed current on a domain $\Omega\subset \mathbb{C}^2$ and $\psi$ a test function on $\Omega$ with $\rm{supp}(\psi) \cap \partial T=\phi$, the sequence $d^{-n}{(H^{n})}^*(\psi T)$ converges to $c\mu^+$ where $c=\int \psi T \wedge \mu^-$. We begin with an analogue of this for skew products of H\'{e}non maps.

\begin{lem}\label{G le1}
For $\lambda\in M$ and for a function $\psi \in C_0(M\times \mathbb{C}^2)$, we have
\begin{equation*}
\frac{1}{d^n}{(H_\lambda^n)}^*(\psi_{\sigma^n(\lambda)} \mu_{\sigma^n(\lambda)}^+ )-\langle
\mu_{\sigma^n(\lambda)},\psi_{\sigma^n(\lambda)}\rangle \mu_\lambda^+ \ra 0
\end{equation*}
as $n\ra \infty$ on $\mathbb{C}^2$, where $\psi_\lambda$ is the restriction of $\psi$ to the fiber $\{\lambda\}\times \mathbb{C}^2$.
\end{lem}

\begin{proof}
Without loss of generality, we consider $\psi \geq 0$. Observe that  
\begin{eqnarray}\label{diff}
&&\Big \lVert \frac{1}{d^n} {(H_\lambda^n)}^*(\psi_{\sigma^n(\lambda)}\mu_{\sigma^n(\lambda)}^+)\Big \rVert-\langle \mu_{\sigma^n(\lambda)},\psi_{\sigma^n(\lambda)}\rangle \nonumber\\ 
&=& \frac{1}{2\pi{d^n}} \int_{\mathbb{C}^2} {(H_\lambda^n)}^* (\psi_{\sigma^n(\lambda)} \mu_{\sigma^n(\lambda)}^+ )\wedge dd^c \log {(1+{\lVert z\rVert}^2)}^{\frac{1}{2}}-\langle \mu_{\sigma^n(\lambda)},\psi_{\sigma^n(\lambda)}\rangle\nonumber\\
&=& \frac{1}{2\pi}\int_{\mathbb{C}^2} (\psi_{\sigma^n(\lambda)} \mu_{\sigma^n(\lambda)}^+ ) \wedge (dd^c {G}_{n,\sigma^n(\lambda)}^-)-\langle \mu_{\sigma^n(\lambda)},\psi_{\sigma^n(\lambda)}\rangle.
\end{eqnarray}
The last equality holds since  $\log {(1+{\lVert z\rVert}^2)}^{\frac{1}{2}}\sim \log^+\lVert z \rVert$ in $\mathbb{C}^2$. The proof of Proposition $1.1$ in \cite{PV} shows that ${G}_{n,\lambda}^-$ converges uniformly to ${G}_\lambda^-$ as $n\ra \infty$ on compact subsets of $\mathbb{C}^2$ and the convergence is uniform in $\lambda$. Thus the last term in (\ref{diff}) which is equal to
\begin{eqnarray}\label{C1}
\frac{1}{2\pi}\int_{\mathbb{C}^2} (\psi_{\sigma^n(\lambda)} \mu_{\sigma^n(\lambda)}^+ ) \wedge (dd^c {G}_{n,\sigma^n(\lambda)}^-) -\frac{1}{2\pi}\int_{\mathbb{C}^2} (\psi_{\sigma^n(\lambda)} \mu_{\sigma^n(\lambda)}^+ ) \wedge (dd^c {G}_{\sigma^n(\lambda)}^-)
\end{eqnarray}
tends to zero as $n\ra \infty$. 

\medskip
 
Consider the currents  
\begin{equation}\label{proof}
T_n= \frac{1}{d^n}{(H_\lambda^n)}^*(\psi_{\sigma^n(\lambda)} \mu_{\sigma^n(\lambda)}^+ )-\langle
\mu_{\sigma^n(\lambda)},\psi_{\sigma^n(\lambda)}\rangle \mu_\lambda^+
\end{equation}
for $n\geq 1$. We will show that $T_n \ra 0$ as $n\ra \infty$. Let $A=\big\{n\geq 0: \langle
\mu_{\sigma^n(\lambda)},\psi_{\sigma^n(\lambda)}\rangle=0 \big\}$. Then it follows by (\ref{C1}) that any subsequence of $\{T_n\}$, which corresponds to the set $A$, tends to $0$ as $n\ra \infty$. So to prove (\ref{proof}), we may assume that $\langle\mu_{\sigma^n(\lambda)},\psi_{\sigma^n(\lambda)}\rangle \neq 0$ for all $n\geq 0$. Now consider the sequence of currents 
$$
C_n=\frac{1}{d^n}{\langle\mu_{\sigma^n(\lambda)},\psi_{\sigma^n(\lambda)}\rangle}^{-1}{(H_\lambda^n)}^*(\psi_{\sigma^n(\lambda)} \mu_{\sigma^n(\lambda)}^+ )
$$
for $n\geq 0$. Let $\gamma$ be a limit point of the sequence $\{C_n\}$. Then by (\ref{C1}), it follows that $\lVert \gamma\rVert=1$. By step $2$ of Theorem $1.3$ in \cite{PV}, it can be shown that $\gamma$ is a closed positive $(1,1)$-current having support in $K_\lambda^+$. So by Corollary \ref{Cor1}, we get that $\gamma=\mu_\lambda^+$. This completes the proof.
\end{proof}

\subsection{Proof of Theorem \ref{thm4}}
\begin{proof}
To show that $H$ is mixing, it is sufficient to prove that
$$
\big \langle \mu, \varphi(\psi\circ H^n) \big \rangle \ra \big \langle \mu, \varphi \big \rangle \big \langle \mu,\psi \big \rangle
$$
as $n\ra \infty$ where $\varphi, \psi \in C^0(M\times \mathbb{C}^2)$. Let $\varphi_\lambda$ be the restriction of $\varphi$ to $\{\lambda\}\times \mathbb{C}^2$ and define $\tilde{\varphi}(\lambda)=\langle \mu_\lambda, \varphi_\lambda\rangle$. Similarly, we define $\psi_\lambda$ and $\tilde{\psi}$. Now
\begin{eqnarray}
\big\langle \mu, \varphi(\psi\circ H^n) \big\rangle 
&=& \int_M \big\langle \mu_\lambda, \varphi_\lambda(\psi_{\sigma^n(\lambda)}\circ H_\lambda^n)\big\rangle \mu'(\lambda) \nonumber \\
&=& \int_M \big\langle {(H_\lambda^n)}^* \mu_{\sigma^n(\lambda)}, \varphi_\lambda(\psi_{\sigma^n(\lambda)}\circ H_\lambda^n)\big\rangle \mu'(\lambda)\nonumber  \\
&=& \int_M \big\langle \mu_{\sigma^n(\lambda)}, \psi_{\sigma^n(\lambda)} {(H_\lambda^n)}_* \varphi_\lambda \big\rangle \mu'(\lambda)\nonumber \\
&=&\int_M \big\langle \mu_{\sigma^n(\lambda)}, \psi_{\sigma^n(\lambda)} \tilde{\varphi}(\lambda)\big\rangle \mu'(\lambda) + \int_M \big\langle \mu_{\sigma^n(\lambda)},\psi_{\sigma^n(\lambda)}\big({(H_\lambda^n)}_*\varphi_\lambda-\tilde{\varphi}(\lambda) \big)\big\rangle \mu'(\lambda)\nonumber\\
&=&\big\langle \mu', \tilde{\varphi}(\tilde{\psi}\circ \sigma^n)\big\rangle + \int_M \big\langle \mu_{\sigma^n(\lambda)},\psi_{\sigma^n(\lambda)}\big({(H_\lambda^n)}_*\varphi_\lambda-\tilde{\varphi}(\lambda) \big)\big\rangle \mu'(\lambda).\label{G3}
\end{eqnarray}
Since $\sigma$ is mixing for $\mu'$, the first term in (\ref{G3}) tends to $\big\langle \mu',\tilde{\varphi}\big\rangle \big\langle \mu',\tilde{\psi} \big\rangle= \big\langle \mu,\varphi\big\rangle \big\langle \mu,\psi \big\rangle$ as $n\ra \infty$. We will show that for each $\lambda\in M$,
\begin{equation}\label{G4}
\big\langle \mu_{\sigma^n(\lambda)}, \psi_{\sigma^n(\lambda)}\big( {(H_\lambda^n)}_*\varphi_\lambda\big)\big\rangle-\big\langle \mu_\lambda,\varphi_\lambda\big\rangle
\big\langle \mu_{\sigma^n(\lambda)},\psi_{\sigma^n(\lambda)}\big\rangle \ra 0
\end{equation} 
as $n\ra \infty$. To do this, observe that for each $\lambda\in M$ and $n\geq 1$, 
\begin{align}
&\big\langle \mu_{\sigma^n(\lambda)}, \psi_{\sigma^n(\lambda)}\big( {(H_\lambda^n)}_*\varphi_\lambda\big)\big\rangle
\nonumber\\
&=\int_{\mathbb{C}^2} \varphi_\lambda {(H_\lambda^n)}^*\big (\psi_{\sigma^n(\lambda)} \mu_{\sigma^n(\lambda)}^+\big) \wedge {(H_\lambda^n)}^*\big(\mu_{\sigma^n(\lambda)}^-\big)\nonumber 
=\frac{1}{2\pi}\int_{\mathbb{C}^2} \varphi_\lambda (d^{-n}){(H_\lambda^n)}^*\big (\psi_{\sigma^n(\lambda)}\mu_{\sigma^n(\lambda)}^+\big) \wedge dd^c {G}_\lambda^- \nonumber \\
&= \frac{1}{2\pi}\int_{\mathbb{C}^2} (dd^c \varphi_\lambda) (d^{-n}){(H_\lambda^n)}^*(\psi_{\sigma^n(\lambda)} \mu_{\sigma^n(\lambda)}^+ )\wedge {G}_\lambda^- 
+\frac{1}{2\pi}\int_{\mathbb{C}^2} (d^c \varphi_\lambda)(d^{-n}) {(H_\lambda^n)}^*(d\psi_{\sigma^n(\lambda)})\wedge \mu_{\sigma^n(\lambda)}^+\wedge {G}_\lambda^- \nonumber\\
&+\frac{1}{2\pi}\int_{\mathbb{C}^2} (d \varphi_\lambda)(d^{-n}) {(H_\lambda^n)}^*(d^c\psi_{\sigma^n(\lambda)})\wedge \mu_{\sigma^n(\lambda)}^+\wedge {G}_\lambda^- 
+\frac{1}{2\pi}\int_{\mathbb{C}^2}  \varphi_\lambda(d^{-n}) {(H_\lambda^n)}^*dd^c (\psi_{\sigma^n(\lambda)})\wedge \mu_{\sigma^n(\lambda)}^+\wedge {G}_\lambda^- .\nonumber \\ \label{G5}
\end{align}
By using an analogue of a result (Theorem $1.6$, (ii)) in \cite{BS3}, it follows that the last three terms of (\ref{G5}) tend to zero as $n\ra \infty$. Hence using Lemma \ref{G le1}, for each $\lambda\in M$, we get 
$$
\big\langle \mu_{\sigma^n(\lambda)}, \psi_{\sigma^n(\lambda)}\big( {(H_\lambda^n)}_*\varphi_\lambda\big) \big\rangle-\frac{1}{2\pi}\big\langle \mu_{\sigma^n(\lambda)}, \psi_{\sigma^n(\lambda)}\big\rangle  \int_{\mathbb{C}^2} (dd^c \varphi_\lambda)(\mu_\lambda^+)\wedge {G}_\lambda^-\ra 0
 $$
 as $n\ra \infty$. Thus for each $\lambda\in M$,
 \begin{equation}
 \big\langle \mu_{\sigma^n(\lambda)}, \psi_{\sigma^n(\lambda)}\big( {(H_\lambda^n)}_*\varphi_\lambda\big)\big\rangle-\big\langle \mu_{\sigma^n(\lambda)}, \psi_{\sigma^n(\lambda)}\big\rangle \big\langle \mu_\lambda,\varphi_\lambda \big\rangle\ra 0\label{G6}
 \end{equation}
 as $n\ra \infty$. Let, 
 $$
 R_n(\lambda):= \big\langle \mu_{\sigma^n(\lambda)}, \psi_{\sigma^n(\lambda)}\big( {(H_\lambda^n)}_*\varphi_\lambda\big) \big\rangle-\big\langle \mu_{\sigma^n(\lambda)}, \psi_{\sigma^n(\lambda)}\big\rangle \big\langle \mu_\lambda,\varphi_\lambda \big\rangle
 $$
for each $n\geq 1$.
By the dominated convergence theorem, it follows by (\ref{G6}) that 
 $$
 \int_M R_n(\lambda)\mu'(\lambda)\ra 0
 $$
 as $n\ra \infty$. This completes the proof.
\end{proof}

\subsection{Proof of Proposition \ref{P4}}
\begin{proof}
For any two compactly supported continuous functions $\varphi$, $\psi$ in $\mathbb{C}^2$  and for a fixed $\lambda\in M$, a similar calculation as in (\ref{G6}) gives 
\[
 \big\langle \mu_{\sigma^n(\lambda)}, \psi\big( {(H_\lambda^n)}_*\varphi\big)\big\rangle-\big\langle \mu_{\sigma^n(\lambda)}, \psi\big\rangle \big\langle \mu_\lambda,\varphi \big\rangle\ra 0
\]
as $n\ra \infty$. Now since for each $\lambda\in M$, ${(H_\lambda)}^*(\mu_{\sigma(\lambda)})=\mu_\lambda$, it follows that
\[
\big\langle \mu_\lambda, \big( {(H_\lambda^n)}^*\psi\big)\varphi \big\rangle-\big\langle \mu_{\sigma^n(\lambda)}, \psi\big\rangle \big\langle \mu_\lambda,\varphi \big\rangle\ra 0
\]
as $n\ra \infty$. This completes the proof.
\end{proof}


\section{Lyapunov exponents}
\subsection{Proof of Theorem \ref{thm6}}
\begin{proof}
Let $v_p=\frac{\partial}{\partial y}\|_p$. Since $\lambda(v,p)\leq \lambda_1$ for all $p\in M$ and for all $v\in T_p$, in particular, we have  
\[
\lambda(v_p)=\lim_{n\ra \infty} \frac{1}{n} \log \lVert DH^n(v_p,p)\rVert\leq \lambda_1
\]
for all $p\in M$. This gives 
\[
\lim_{n\ra \infty} \frac{1}{n} \int \log \lVert DH^n(v_p)\rVert\leq \lambda_1
\]
where $DH^n(v_p)$ is same as $DH^n(p,v_p)$. We shall show that 
\[
\lim_{n\ra \infty} \frac{1}{n} \int \log \lVert DH^n(v_p)\rVert\geq \log d 
\]
which in turn provides a lower bound for the largest Lyapunov exponent $\lambda_1$.

\medskip

Let $X=\{x=0\}$ in $\mathbb{C}^2$ and for a $\lambda\in M$, $K_{\lambda,0}= X\cap K_\lambda^+$. Now note that
\begin{eqnarray*}
\int \frac{1}{n} \log \lVert DH^n(v_p)\rVert d\mu
&=& \int \Big(\int \frac{1}{n} \log \lVert DH_\lambda^n(v_p)\rVert d\mu_\lambda \Big)d\mu'.
\end{eqnarray*} 
A result analogous to Theorem \ref{thm1} shows that  
$$
\lim_{k\ra \infty} d^{-k} \big(H_{\sigma^{-k}(\lambda)}^{-1}\circ \cdots \circ H_{\sigma^{-1}(\lambda)}^{-1}\big)[X]=\mu_\lambda^-
$$
for each $\lambda\in M$. Thus
\begin{align}
&\int \frac{1}{n} \log \lVert DH_\lambda^n(v_p)\rVert d\mu_\lambda
=\lim_{k\ra \infty} \frac{1}{n} \int \log \lVert DH_\lambda^n(v_p)\rVert\mu_\lambda^+ \wedge d^{-k}{\big(H_{\sigma^{-k}(\lambda)}^{-1}\circ \cdots \circ H_{\sigma^{-1}(\lambda)}^{-1}\big)}^*[X] \nonumber\\
&\geq \lim_{k\ra \infty} \frac{1}{n} \int \log \lvert \partial_y{(DH_\lambda^n)}_2 \rvert \mu_\lambda^+  \wedge d^{-k}{\big(H_{\sigma^{-k}(\lambda)}^{-1}\circ \cdots \circ H_{\sigma^{-1}(\lambda)}^{-1}\big)}^*[X] \nonumber\\
&=\lim_{k\ra \infty} \frac{1}{n} \int \log \lvert \partial_y{(DH_\lambda^n)}_2 \rvert {\big(H_{\sigma^{-k}(\lambda)}^{-1}\circ \cdots \circ H_{\sigma^{-1}(\lambda)}^{-1}\big)}^*(\mu_{\sigma^{-k}(\lambda)}^+) \wedge {\big(H_{\sigma^{-k}(\lambda)}^{-1}\circ \cdots \circ H_{\sigma^{-1}(\lambda)}^{-1}\big)}^*[X]\nonumber\\
&=\lim_{k\ra \infty}\frac{1}{n} \int \log \lvert \partial_y{(DH_\lambda^n)}_2 \circ \big(H_{\sigma^{-1}(\lambda)}^{-1}\circ \cdots \circ H_{\sigma^{-k}(\lambda)}^{-1}\big) \rvert \mu_{\sigma^{-k}(\lambda)}^+\wedge [X].\nonumber\\
&= \lim_{k\ra \infty}\frac{1}{n} \int \log \lvert \partial_y{(DH_\lambda^n)}_2 \circ \big(H_{\sigma^{-1}(\lambda)}^{-1}\circ \cdots \circ H_{\sigma^{-k}(\lambda)}^{-1}\big) \rvert \mu_{K_{\sigma^{-k}(\lambda),0}^+}\nonumber\\
\label{lya}
\end{align}
for each $\lambda\in M$.
Now note that 
$$
\partial_y{(DH_\lambda^n)}_2 \circ \big(H_{\sigma^{-1}(\lambda)}^{-1}\circ \cdots \circ H_{\sigma^{-k}(\lambda)}^{-1}\big)=d^n y^{(d^n-1)d^k}+\cdots$$
where the dots represent terms of lower degree, and thus 
$$
\log \big \lvert \partial_y{(DH_\lambda^n)}_2 \circ \big(H_{\sigma^{-1}(\lambda)}^{-1}\circ \cdots \circ H_{\sigma^{-k}(\lambda)}^{-1}\big)\big\rvert=n\log d+ \log \lvert p_k \rvert 
$$
where $p_k$ is a monic polynomial in $y$. Therefore by (\ref{lya})
\begin{eqnarray*}
\int \frac{1}{n} \log \lVert DH^n(v_p)\rVert d\mu &=&\int \Big(\int \frac{1}{n} \log \lVert DH_\lambda^n(v_p)\rVert d\mu_\lambda \Big) d\mu'\\
&\geq & \log d.
\end{eqnarray*}

\no
Now choosing an appropriate vector in the tangent space at the point $p$, we get that $\lambda_1\geq \lambda_\sigma$ where $\lambda_\sigma$ is the largest Lyapunov exponent of $\sigma$. Thus 
\[
\lambda_1\geq \max\{\log d, \lambda_\sigma\}.
\]
\end{proof}


\section{Skew products of H\'{e}non maps: Entropy bounds}
\no
Before proving Theorem \ref{thm5}, we will first show that 
\[
h_{\rm{top}}(H)\geq h_{\rm{top}}(\sigma)+\log d
\]
by adapting Smillie's arguments in \cite{Sm} that are valid for a single H\'{e}non map. The main difficulty of obtaining bounds for volume growth of certain discs requires some ideas from \cite{YK}.

\medskip

\no
Let $\tau: \mathbb{C}\rightarrow \mathbb{C}^2$ be defined by $\tau(z)=(0,z)$. For $\lambda\in M$ and for $n\geq 1$, consider the following metrics on $\mathbb{C}^2$: 
$$
e_{n,\lambda}(x,y)=\max_{0\leq i\leq n-1}e(H_{\sigma^{i}(\lambda)}\circ \cdots \circ H_\lambda(x),H_{\sigma^{i}(\lambda)}\circ \cdots \circ H_\lambda(y)),$$
where $e$ is the Euclidean metric on $\mathbb{C}^2$. Further, for each $\lambda\in M$ and for $n\geq 1$, consider
$$
V_{n,\lambda}= H_\lambda^{-1}\circ \cdots \circ H_{\sigma^{n-1}(\lambda)}^{-1}(V_R)\cap V_R
$$
and let
$$
\upsilon_{n,\lambda}=\text{Area of } \big(H_{\sigma^{-1}(\lambda)}\circ \cdots \circ H_{\sigma^{-n}(\lambda)}(\tau (D_R))\cap V_R\big).
$$

\medskip

\no
Let 
\[
\bigcup_{i=1}^{K(n,\epsilon,V_{n,\lambda})}U_i^\lambda
\]
be a minimal $\epsilon$-covering of $V_{n,\lambda}$ with respect to $e_{n,\lambda}$-metric.
Thus
\[
V_R \cap (H_{\sigma^{n-1}(\lambda)}\circ\cdots \circ H_\lambda)(V_R)= \bigcup_{i=1}^{K(n,\epsilon,V_{n,\lambda})}(H_{\sigma^{n-1}(\lambda)}\circ\cdots \circ H_\lambda)(U_i^\lambda)
\]
which implies
\begin{equation}\label{vol}
\upsilon _{n,\sigma^n(\lambda)}\leq K(n,\epsilon,V_{n,\lambda}) \upsilon_{\sigma^n(\lambda)}^{0}(n,\epsilon)
\end{equation}
for all $\lambda\in M$ and for all $n\geq1$, where
$$
\upsilon_{\sigma^n(\lambda)}^{0}(n,\epsilon)=\sup_{1\leq i \leq K(n,\epsilon,V_{n,\lambda})}\text{Area of }(H_{\sigma^{n-1}(\lambda)}\circ\cdots \circ H_\lambda)(U_i^\lambda\cap \tau(D_R)).
$$

\medskip

\no
{\it Step 1:}
We will prove the following which is a variant of the main result given \cite{YK}.

\medskip

\no
{\it
{For any sequence ${\{\lambda_k\}}_{k\geq 1}\subset M$, 
$$
\lim_{\epsilon\ra 0}\limsup_{k\ra \infty}\frac{1}{n_k}\upsilon_{\sigma^{n_k}(\lambda_k)}^{0}(n_k,\epsilon)\ra 0,
$$
whenever $n_k\ra \infty$ as $k\ra \infty$.}}

\medskip

\no
Fix $p\geq 1$. Let ${\{F_n\}}_{n\geq 1}$ be defined by
$$
F_n=H_{\lambda_1^{(n)}}\circ\cdots\circ H_{\lambda_p^{(n)}}
$$ 
where $\{\lambda_1^{(n)}, \lambda_2^{(n)},\cdots,\lambda_p^{(n)}\}\subset M$. For each $n\geq 1$, let 
\[
e_n(x,y)=\max_{0\leq i \leq n} e(F_n\circ \cdots \circ F_1(x),F_n\circ \cdots \circ F_1(y)) 
\] 
and 
\[
W_n=F_1^{-1}\circ \cdots \circ F_n^{-1}(V_R)\cap V_R.
\]  
Let 
$$
\bigcup_{i=1}^{m_n}U_i
$$
be a minimal $\epsilon$-covering of $W_n$ with respect to the metric $e_n$.

\medskip

\no
Consider an $\epsilon < L^{-p}$ where
\[
{\lVert DH_\lambda\rVert}_{V_R}\leq L < \infty
\] 
for all $\lambda\in M$. This choice of $L$ is possible since $\lambda$ varies in a compact metric space. Next fix some $U_i$ where $1\leq i \leq m_n$ which we call $U$ for the sake of notational simplicity. Let $x_0\in V_R$ be the center of $U$. Consider the orbit $\{x_0,x_1,...,x_n\}$ where 
$$
x_i=(F_i \circ\cdots \circ F_1)(x_0)\in V_R
$$
for $1\leq i \leq n$. Let $U^i$ be the balls of radius $\epsilon$ centered at $x_i$ with respect to the standard metric on $\mathbb{C}^2$. Define $\psi_i(x)=x_i+\epsilon x$ and consider the new mappings
$$
G_i=\psi_i^{-1}\circ F_i \circ \psi_{i-1}
$$  
for $1\leq i \leq n$. Note that the $G_i$'s are well-defined on all of $\mathbb{C}^2$ and in particular on 
$$
\mathcal{B}_i=\psi_{i-1}^{-1}\big((F_i^{-1}\circ \cdots \circ F_n^{-1})(V_R)\cap (F_{i-1}\circ \cdots \circ F_1)(V_R)\big)\subset \psi_{i-1}^{-1}(V_R).
$$ 

\medskip

\no
Let 
\[
\upsilon_0(n,\epsilon)=\text{Area of }(F_n\circ\cdots \circ F_1)(U\cap \tau(D_R))=\text{Area of }(\psi_n \circ G_n\circ\cdots \circ G_1\circ \psi_0^{-1})(U\cap \tau(D_R))
\]
for $n\geq 1$. Fix $l\geq 1$. By arguing as in Proposition $2.1$ in \cite{YK}, we can show that 
\begin{equation}\label{volume}
\upsilon_0(n,\epsilon)\leq c(\tau,\epsilon){d(l)}^{n} {(M_l^{\mathcal{B}_1}(G_1))}^{\frac{4}{l}}\cdots {(M_l^{\mathcal{B}_n}(G_{n}))}^{\frac{4}{l}}
\end{equation}
where 
\[
M_l^\mathcal{B}(f)=\max \{1,{\lVert Df \rVert}_{C^l}^\mathcal{B}\}
\]
with the $C^l$-norm ${\lVert Df \rVert}_{C^l}^\mathcal{B}=\max_{1\leq s \leq l} \lVert D^s f\rVert_\mathcal{B}$ and $d(l)$ is a constant depending only on $l$.

\medskip

\no
Now note that 
\begin{equation*}
{\lVert D^s G_i\rVert}_{\mathcal{B}_i}\leq \epsilon^{s-1}{\lVert D^s F_i \rVert}_{\mathcal{D}_i}
\end{equation*}
where $\mathcal{D}_i={(F_i^{-1}\circ \cdots \circ F_n^{-1})(V_{R})\cap (F_{i-1}\circ\cdots\circ F_1)(V_R)}$.

\medskip

\no
Since $\epsilon < L^{-p}$, 
\[
{M_l^{\mathcal{B}_i}(G_i)}\leq \max \{1, {\lVert D F_i\rVert}_{\mathcal{D}_i}\} \leq \max \{1, L^p \}
\] 
for $1\leq i \leq n$.

\medskip

\no
As in \cite{YK}, for $\epsilon<L^{-p}$, it follows that
\begin{eqnarray}\label{volest}
\frac{1}{n_k} \log \upsilon_0(n_k,\epsilon)\leq \frac{1}{n_k}\log c(\tau,\epsilon)+\log d(l)+ \frac{4}{l} p\log L. 
\end{eqnarray}
Note that the volume estimate in (\ref{volest}) is independent of the sequence $\{{F_n}\}_{n\geq 1}$ we start with. 

\medskip

\no
Consequently, 
\begin{eqnarray}
\frac{1}{n_k} \log \Big(\text{Area of }{( H_{\sigma^{n_k-1}(\lambda_k)}\circ\cdots\circ H_{\lambda_k})}^p(U\cap \tau(D_R)\Big) &\leq& \frac{1}{n_k}\log c(\tau,\epsilon)+\log d(l)+ \frac{4}{l} p \log L \nonumber\\
&&
\end{eqnarray}
since we can rewrite ${( H_{\sigma^{n_k-1}(\lambda_k)}\circ\cdots\circ H_{\lambda_k})}^p$ as the composition of $n_k$-many mappings each of which is composition of $p$-many $H_\lambda$'s. 

\medskip

\no
Thus for $\epsilon<{1}/{L^p}$, we have
\begin{eqnarray*}
\frac{1}{n_k}\log \upsilon_{\sigma^{n_k}(\lambda_k)}^0 (n_k,\epsilon)&\leq& \frac{1}{p n_k}\log c(\tau,\epsilon)+\frac{\log d(l)}{p}+ \frac{4\log L}{l}  
\end{eqnarray*}
which gives
\begin{equation*}
\lim_{\epsilon\ra 0}\limsup_{k\ra \infty}\frac{1}{n_k}\log \upsilon_{\sigma^{n_k}(\lambda_k)}^0 (n_k,\epsilon)\leq \frac{4\log L}{l}.
\end{equation*}
As $l\ra \infty$, we get
\begin{equation}\label{limit}
\lim_{\epsilon\ra 0}\limsup_{k\ra \infty}\frac{1}{n_k}\log \upsilon_{\sigma^{n_k}(\lambda_k)}^0 (n_k,\epsilon)=0.
\end{equation}

\medskip

\no
{\it Step 2:}
{\it{Claim: For $0<\alpha<1$, there exists an ${\epsilon_\alpha}> 0$ and ${n_{\alpha}}\geq 1$ such that
\[
K(n,\epsilon,V_{n,\lambda})\geq d^{\alpha n}
\]
for all $n\geq n_\alpha$, for all $\epsilon\leq {\epsilon_\alpha}$ and for all $\lambda\in M$.}}

\medskip

\no
If not, then for any given $\epsilon>0$, we can choose a sequence of natural numbers ${\{n_k\}}_{k\geq 1}$ and ${\{\lambda_k\}}_{k\geq 1}\subseteq M$ such that
\[
K(n_k,\epsilon,V_{n_k,\lambda_k})<d^{\alpha n_k}.
\]
By (\ref{vol}), we get
\begin{equation}
\upsilon_ {n_k,\sigma^{n_k}(\lambda_k)}\leq K(n_k,\epsilon,V_{n_k,\lambda_k})\upsilon_{\sigma^{n_k}(\lambda_k)}^{0}(n_k,\epsilon)
\end{equation}
for all $k\geq 1$. Thus
\begin{equation}
\limsup_{k\ra \infty}\frac{1}{n_k}\log K(n_k,\epsilon,V_{n_k,\lambda_k})+\limsup_{k\ra \infty}\frac{1}{n_k} \log \upsilon_{\sigma^{n_k}(\lambda_k)}^{0}(n_k,\epsilon) \geq \log d> \alpha \log d+ \delta_\alpha
\end{equation}
where $\delta_\alpha$ is a positive quantity depending on $\alpha$. 
Now by (\ref{limit}), a sufficiently small $\epsilon_\alpha>0$ can be chosen such that
\[
\limsup_{k\ra \infty}\frac {1}{n_k}\log\upsilon_{\sigma^{n_k}(\lambda_k)}^{0}(n_k,\epsilon)< \frac{\delta_\alpha}{2}
\]
for all $\epsilon\leq \epsilon_\alpha$. Hence
\begin{equation}
\limsup_{k\ra \infty}\frac{1}{n_k}\log K(n_k,\epsilon,V_{n_k,\lambda_k})>\alpha\log d+\frac{\delta_\alpha}{2}
\end{equation}
for all $\epsilon\leq \epsilon_\alpha$. This completes the proof.

\medskip
   
\no
{\it Step 3:}
Since for all $n\geq 1$ and $\epsilon>0$, the cardinality of an $(n,\epsilon)$-separated set of $V_{n,\lambda}$, which we denote by $S(n,\epsilon,V_{n,\lambda})$, is at least $K(n,\epsilon,V_{n,\lambda})$, we have the following: for $0<\alpha<1$, there exists an ${\epsilon_\alpha}> 0$ and ${n_{\alpha}}\geq 1$ such that
\begin{equation}\label{estimate}
S(n,\epsilon,V_{n,\lambda})\geq d^{\alpha n}
\end{equation}
for all $n\geq n_\alpha$, for all $\epsilon\leq {\epsilon_\alpha}$ and for all $\lambda\in M$. 

\medskip

Let $E=\{\lambda_1,\ldots, \lambda_r\}$ be an $(n,\epsilon)$-separated set of $M$ and $Z_{\lambda_i}$ be an $(n,\epsilon)$-separated set of $V_{n,\lambda_i}$ for $1\leq i \leq r$. Clearly 
$$
Z=\cup_{i=1}^r \{\lambda_i\}\times Z_{\lambda_i}
$$
is an $(n,\epsilon)$-separated set of $V_n$. Hence
\[
S(n,\epsilon, V_n)\geq \sum_{i=1}^r S(n,\epsilon,V_{n,\lambda_i})
\]
where
$$
V_n=\cup_{\lambda\in M}\{\lambda\}\times V_{n,\lambda}.
$$
Now if we take $\epsilon_\alpha>0$ sufficiently small and $n_\alpha\geq 1$ large enough, then it follows from (\ref{estimate})
that 
\begin{equation}\label{separate}
S(n,\epsilon, V_n)\geq S(n,\epsilon,M)d^{\alpha n}
\end{equation}
for all $\epsilon\leq \epsilon_\alpha$ and for all $n\geq n_\alpha$.

\medskip

Let $Q$ denote the quotient space $\big(M\times(V_R\cup V_R^+)\big)/M\times V_R^+$ and let $q$ correspond to the class of $M\times V_R^+$ in the quotient space, i.e., $q$ is a point in $Q$. Consider the natural metric $\bar{e}$ on $Q$, defined as follows:
\begin{eqnarray*}
\bar{e}(x,y)&=& \min\{ e(x,y), e(x,M\times V_R^+)+e(y,M\times V_R^+)\}\\
\bar{e}(x,q)&=&e(x,M\times V_R^+)
\end{eqnarray*}
where $e$ is the natural Euclidean metric on $\mathbb{C}^2$. Define $\overline{H}: Q\to  Q$ by 
\[
z\mapsto
\begin{cases}
(\sigma(\lambda),{h_\lambda}(x,y)); &\quad \text{if } z=(\lambda,x,y)\in M\times V_R\\
q; &\quad \text{otherwise}
\end{cases}
\]

\medskip

For a fixed $\epsilon>0$, we can choose sufficiently large $n$ such that $V_n\subset \subset M\times {V_R}$ and on these $V_n$'s the quotient metric $\bar{e}$ on $Q$ coincides with the natural metric on $M\times V_R$. Hence  
\begin{eqnarray*}
h(\overline{H})&=&\lim_{\epsilon\ra 0}\limsup_{n\ra \infty}\frac{1}{n} \log S(n,\epsilon, Q)\\
&\geq&\lim_{\epsilon\ra 0}\limsup_{n\ra \infty}\frac{1}{n} \log S(n,\epsilon,V_n).
\end{eqnarray*}
Now it follows from (\ref{separate}) that 
\begin{eqnarray*}
h(\overline{H})&\geq& \lim_{\epsilon\ra 0}\limsup_{n\ra \infty}\frac{1}{n} S(n,\epsilon,M)+\alpha \log d\\
&\geq& h(\sigma)+\alpha \log d.
\end{eqnarray*}
Letting $\alpha\rightarrow 1$, as explained in Theorem $1$ in \cite{Sm}, we conclude that 
\[
h(H)\geq h(\sigma)+\log d.
\]

\subsection{Proof of Theorem \ref{thm5}}
\begin{proof}
Let $L$ be a $C^2$-smooth subharmonic function of $\lvert y \rvert$ such that $L(y)=\log \lvert y \rvert$ for $\lvert y \rvert >R$ and define $\Theta=\frac{1}{2\pi}dd^c L$ and $\Theta_{n,\lambda}={(H_\lambda^n)}^* \Theta$ for each $\lambda\in M$ and for each $n\geq 1$. Now consider the disc $\mathcal{D}= \{x=0,\vert y \vert < R\} \subset \mbb C^2$ 
and let 
\[
\alpha_{n,\la}= [\mathcal{D}] \wedge \Theta_{n,\lambda}
\]
for each $\lambda\in M$ and for each $n\geq 1$. Note that $\int_{\mathcal{D}} \alpha_{n,\lambda}=d^n$ and thus $\rho_{n,\lambda}=d^{-n} \alpha_{n,\lambda}$ is a probability measure. 

\medskip

\no
{\it Step 1:} For each $\lambda\in M$ and $n\geq 1$, let
\[
\mu_{n,\lambda} =\frac{1}{n} \sum_{j=0}^{n-1} \frac{d^{-n}}{2\pi} {\big(H_{\sigma^{-j}(\lambda)}^{-1} \circ \cdots \circ H_{\sigma^{-1}(\lambda)}^{-1}\big)}^*(\alpha_{n,\sigma^{-j}(\lambda)}).
\]
 {\it Claim:}  For each $\lambda\in M$, 
$
\mu_{n,\lambda}\ra \mu_\lambda
$
as $n\ra \infty$.

\medskip

Let $\{l_n\}_{n\geq 1}$ be a sequence of natural numbers such that $l_n \ra \infty$ and ${l_n}/{n}\ra 0$ as $n\ra \infty$. Then for each $n\geq 1$ and for each $\lambda\in M$, $\mu_{n,\lambda}$ is equal to the following:
\begin{multline} \label{eq a}
\frac{d^{-n}}{2n\pi} \Big( \sum_{j=0}^{l_n} {\big(H_{\sigma^{-j}(\lambda)}^{-1} \circ \cdots \circ H_{\sigma^{-1}(\lambda)}^{-1}\big)}^*\big(\left[\mathcal{D}\right]\wedge \Theta_{n,\sigma^{-j}(\lambda)}\big)\\
 + \sum_{j=l_n+1}^{n-l_n}{\big(H_{\sigma^{-j}(\lambda)}^{-1} \circ \cdots \circ H_{\sigma^{-1}(\lambda)}^{-1}\big)}^*\big(\left[ \mathcal{D}\right]\wedge \Theta_{n,\sigma^{-j}(\lambda)}\big)\\ 
 + \sum_{j=n-l_n+1}^{n-1}{\big(H_{\sigma^{-j}(\lambda)}^{-1} \circ \cdots \circ H_{\sigma^{-1}(\lambda)}^{-1}\big)}^*\big(\left[\mathcal{D}\right]\wedge \Theta_{n,\sigma^{-j}(\lambda)}\big)\Big). 
\end{multline}
We show that the first and the third terms of the above equation tend to zero as $n\ra \infty$. Let $\chi$ be a test function. Then consider
\begin{multline}
{d^{-n}}\Big\langle  \sum_{j=0}^{l_n} {\big(H_{\sigma^{-j}(\lambda)}^{-1} \circ \cdots \circ H_{\sigma^{-1}(\lambda)}^{-1}\big)}^* \big(\left[\mathcal{D}\right]\wedge\Theta_{n,\sigma^{-j}(\lambda)}\big)\\
+\sum_{j=n-l_n+1}^{n-1}{\big(H_{\sigma^{-j}(\lambda)}^{-1} \circ \cdots \circ H_{\sigma^{-1}(\lambda)}^{-1}\big)}^*\big(\left[\mathcal{D}\right]\wedge\Theta_{n,\sigma^{-j}(\lambda)}\big),\chi\Big\rangle \nonumber 
\end{multline}
which is equal to 
\begin{multline}
\sum_{j=0}^{l_n}\Big\langle {d^{-n}}{\big(H_{\sigma^{-j}(\lambda)}^{-1} \circ \cdots \circ H_{\sigma^{-1}(\lambda)}^{-1}\big)}^* \big(\left[\mathcal{D}\right]\wedge\Theta_{n,\sigma^{-j}(\lambda)}\big), \chi\Big\rangle\\
 + \sum_{j=n-l_n+1}^{n-1} \Big\langle{d^{-n}} {\big(H_{\sigma^{-j}(\lambda)}^{-1} \circ \cdots \circ H_{\sigma^{-1}(\lambda)}^{-1}\big)}^*\big(\left[\mathcal{D}\right]\wedge\Theta_{n,\sigma^{-j}(\lambda)}\big) 
,\chi\Big\rangle. \label{eq b}
\end{multline}
\no
Note that each term of (\ref{eq b}) is uniformly bounded by some fixed constant independent of $n$ and the total number of terms in (\ref{eq b}) is $2 l_n$. This shows that the first and the third terms in (\ref{eq a}) tend to $0$ as $n\ra \infty$ since ${l_n}/{n}\ra 0$ as $n\ra \infty$. \\

\medskip

Furthermore,
\begin{eqnarray*}
 &&d^{-n} {\big(H_{\sigma^{-j}(\lambda)}^{-1} \circ \cdots \circ H_{\sigma^{-1}(\lambda)}^{-1}\big)}^* \big(\left[\mathcal{D}\right]\wedge \Theta_{n,\sigma^{-j}(\lambda)}\big)\\
 &&= d^{-j}{\big(H_{\sigma^{-j}(\lambda)}^{-1} \circ \cdots \circ H_{\sigma^{-1}(\lambda)}^{-1}\big)}^* \left[\mathcal{D}\right] \wedge (d^{-n+j}){\big(H_{\sigma^{n-j-1}(\lambda)}\circ \cdots \circ H_\lambda\big)}^*\big(\frac{1}{2\pi} dd^c L\big)\\
 &&= \frac{\mu_{j,\lambda}}{2\pi}^- \wedge dd^c {G}_{n-j,\lambda}^+
 \end{eqnarray*}
 where $\mu_{j,\lambda}^-=d^{-j}{\big(H_{\sigma^{-j}(\lambda)}^{-1} \circ \cdots \circ H_{\sigma^{-1}(\lambda)}^{-1}\big)}^* \left[\mathcal{D}\right]$.\\
\indent 
 Now note that
 \begin{eqnarray}
 \frac{1}{n }\Big(\sum_{j=l_n+1}^{n-l_n} {G}_{n-j,\lambda}^+ \mu_{j,\lambda}^-\Big) 
&=& {G}_\lambda^+ \frac{1}{n} \sum_{j=l_n+1}^{n-l_n}\mu_{j,\lambda}^- + \frac{1}{n} \sum_{j=l_n+1}^{n-l_n}\Big( {G}_{n-j,\lambda}^+ -{G}_\lambda^+\Big)\mu_{j,\lambda}^- \nonumber\\
\label{eq c}
 \end{eqnarray}
 \no
 and since on any compact subset of $\mathbb{C}^2$, the sequence of functions ${G}_{n-j,\lambda}^-$ converge to ${G}_\lambda^-$ as $n-j\ra \infty$ and $\mu_{j,\lambda}^-$ has uniformly bounded mass, the last term of (\ref{eq c}) tends to zero as $n\ra \infty$. Since $\left[ \mathcal{D}\right]$ is a closed positive $(1,1)$-current of mass $1$ on $\mathbb{C}^2$ vanishing outside $\mathcal{D}$, an analogous result to Theorem $1.6$, (iii) in \cite{BS3} implies that $\mu_{j,\lambda}^-$ converges to $\mu_\lambda^-$ as $j\ra \infty$. Therefore
 \[
\lim_{n\ra \infty} \frac{d^{-n}}{n}\sum_{j=l_n+1}^{n-l_n}{\big(H_{\sigma^{-j}(\lambda)}^{-1} \circ \cdots \circ H_{\sigma^{-1}(\lambda)}^{-1}\big)}^* \big(\left[\mathcal{D}\right]\wedge \Theta_{n,\sigma^{-j}(\lambda)}\big)=\mu_\lambda
 \]
 for each $\lambda\in M$ and thus
\[
\lim_{n\ra \infty}\mu_{n,\lambda}=\mu_\lambda
\] 
for each $\lambda\in M$.

\medskip

\no
{\it Step 2:} For an arbitrary compactly supported probability measure $\mu'$ on $M$ and for each $n \ge 0$, let $\mu_n$ and $\rho_n$ be defined by the following recipe, i.e., for a test function $\varphi$,
\begin{equation} \label{mu}
\langle \mu_n, \varphi \rangle = \int_M \left ( \int_{\{ \la \} \times \mbb C^2} \varphi \; \mu_{n, \la}  \right) \mu'(\la) \;\; \text{and} \;\; \langle \rho_n, \varphi \rangle  = \int_M \left (
\int_{\{ \la \} \times \mbb C^2} \varphi \; \rho_{n, \la} \right) \mu'(\la). 
\end{equation}

\medskip

\no
{\it Claim:} 
$$
\lim_{n\ra \infty} \mu_n=\mu \; \text{and} \; \mu_n=\frac{1}{n}\sum_{j=0}^{n-1}H_*^j\rho_n
$$
where $H$ is as in (\ref{homeo}).  For the first claim, note that for all test functions $\varphi$
\begin{eqnarray}
 \lim_{n\ra \infty}\langle \mu_n,\varphi\rangle &=& \lim_{n\ra \infty}\int_M \langle \mu_{n,\la},\varphi\rangle \mu'(\la) = \int_M \lim_{n\ra \infty}\langle \mu_{n,\la},\varphi\rangle \label{28}
\mu'(\la)\\ \notag
 &=& \int_M \langle \mu_\la,\varphi\rangle \mu'(\la) = \langle\mu,\varphi\rangle  
\end{eqnarray}
where the second equality follows by the dominated convergence theorem. For the second claim, note that
\begin{eqnarray}
\left \langle \frac{1}{n}\sum_{j-0}^{n-1}H^j_*\rho_n,\varphi \right \rangle 
&=& \int_M \left \langle \frac{1}{n} \sum_{j=0}^{n-1} {\big(H_{\sigma^{-j}(\lambda)}^{-1} \circ \cdots \circ H_{\sigma^{-1}(\lambda)}^{-1}\big)}^*(\rho_{n,\sigma^{-j}(\lambda)}),\varphi \right \rangle \mu'(\la)\nonumber\\
 &=& \int_M \langle\mu_{n,\la},\varphi \rangle \mu'(\la) = \langle \mu_n,\varphi\rangle.
\end{eqnarray}
Hence we get
\[
\lim_{n\ra\infty}\frac{1}{n}\sum_{j-0}^{n-1}H^j_*\rho_n=\mu.
\]

\medskip

\no
{\it Step 3:} We choose partition $\mathcal{P}=\{\mathcal{P}_1,\ldots,\mathcal{P}_k\}$ of $V_R$ such that
 $\int \mu_\lambda(\partial \mathcal{P}_i)d\mu'(\lambda)=0$ for $1\leq i \leq k$ where $\partial \mathcal{P}_i $ denotes the boundary of $\mathcal{P}_i$.  Let $\mathcal{B}$ be a member of $\bigvee_{i=0}^{n-1}{(H_\lambda^i)}^{-1}\mathcal{P}$. Then
 $$
 \rho_{n,\lambda}(\mathcal{B})=d^{-n}\int_{\mathcal{B}\cap \mathcal{D}}H_\lambda^{n*}\Theta=d^{-n}\int_{H_\lambda^n(\mathcal{B}\cap \mathcal{D})}\Theta.
 $$
 Since $\Theta$ is bounded above on $\mathbb{C}^2$, there exists $C>0$ such that 
 \[
 \rho_{n,\lambda}(\mathcal{B})\leq Cd^{-n}\text{Area} (H_\lambda^n(\mathcal{B}\cap \mathcal{D})).
 \]
 If $v_0(H,n,\lambda,\epsilon)$ denotes the supremum of the area of $H_\lambda^n(\mathcal{B}\cap \mathcal{D})$ over all $\epsilon$-balls (in the $d_{n,\lambda}$ metric where $d_{n,\lambda}(x,y)=\max_{0\leq i \leq (n-1)}d(H_{\sigma^{i-1}(\lambda)}\circ\cdots \circ H_\lambda(x),H_{\sigma^{i-1}(\lambda)}\circ\cdots \circ H_\lambda(y))$) $\mathcal{B}$ of $\bigvee_{i=0}^{n-1}{(H_\lambda^i)}^{-1}\mathcal{P}$, then
 \begin{equation}\label{rho}
 H_{\rho_{n,\lambda}}\Big(\bigvee_{i=0}^{n-1}{(H_\lambda^i)}^{-1}\mathcal{P}\Big)\geq -\log C+n\log d-\log v_0(H,n,\lambda,\epsilon).
 \end{equation}
 Let $\mathcal{F}_R$ be the $\sigma$-algebra formed by all sets $\mathcal{A}\times V_R$ where $\mathcal{A}\in \mathcal{F}$. Further, let $\mathcal{Q}=\{\mathcal{Q}_1,\ldots,\mathcal{Q}_k\}$ where $\mathcal{Q}_i=M\times\mathcal{P}_i$ for $1\leq i \leq k$. Then by (\ref{mu}) and (\ref{rho})
 \[
 H_{\rho_n}\Big(\bigvee_{i=0}^{n-1}H^{-i}\mathcal{Q}|\mathcal{F}_{R}\Big)\geq -\log C+n\log d-\int_M \log v_0(H,n,\lambda,\epsilon)d\mu'(\lambda).
 \]
 Following the same method as explained in \cite{K}, we get
  \begin{equation}
 \lim_{n\ra\infty}\frac{1}{n}H_\mu \Big(\bigvee_{i=0}^{n-1}H^{-i}\mathcal{Q}|\mathcal{F}_R\Big)\geq \log d-\limsup_{n\ra \infty}\frac{1}{n}\int_M \log v_0(H,n,\lambda,\epsilon).
 \end{equation}
 This shows 
 \[
 \lim_{n\ra\infty}\frac{1}{n}\int H_{\mu_\lambda}\Big(\bigvee_{i=0}^{n-1}{(H_\lambda^i)}^{-1}\mathcal{P}\Big)d\mu'(\lambda)\geq \log d-\limsup_{n\ra \infty}\frac{1}{n}\int_M \log v_0(H,n,\lambda,\epsilon).
 \]
By the Abramov-Rohlin theorem \cite{AR}, and letting $\epsilon\ra 0$, we get
\[
h(H)\geq h(\sigma)+ \log d
\] 
since 
\[
\lim_{\epsilon\ra 0}\limsup_{n\ra \infty}\frac{1}{n}\int_M \log v_0(H,n,\lambda,\epsilon)=0.
\]
\end{proof}


\end{document}